\documentclass[11pt]{amsart}

\pdfoutput=1

\usepackage[text={400pt,660pt},centering]{geometry}

\usepackage{esint,amssymb}
\usepackage{graphicx}
\usepackage{MnSymbol}
\usepackage{mathtools}
\usepackage[colorlinks=true, pdfstartview=FitV, linkcolor=blue, citecolor=blue, urlcolor=blue,pagebackref=false]{hyperref}
\usepackage{microtype}

\usepackage{bm}
\usepackage{dsfont}

\newtheorem{proposition}{Proposition}
\newtheorem{theorem}[proposition]{Theorem}
\newtheorem{lemma}[proposition]{Lemma}

\theoremstyle{remark}
\newtheorem{remark}[proposition]{Remark}

\theoremstyle{definition}
\newtheorem{definition}[proposition]{Definition}

\numberwithin{equation}{section}
\numberwithin{proposition}{section}
\numberwithin{figure}{section}
\numberwithin{table}{section}
\newcommand{\N}{\mathbb{N}}

\newcommand{\R}{\mathbb{R}}

\newcommand{\E}{\mathbb{E}}
\renewcommand{\P}{\mathbb{P}}

\newcommand{\ep}{\varepsilon}

\renewcommand{\le}{\leqslant}
\renewcommand{\ge}{\geqslant}

\renewcommand{\subset}{\subseteq}
\newcommand{\la}{\left\langle}
\newcommand{\ra}{\right\rangle}

\newcommand{\Ll}{\left}
\newcommand{\Rr}{\right}
\renewcommand{\d}{\mathrm{d}}

\DeclareMathOperator{\tr}{tr}

\renewcommand{\bar}{\overline}

\newcommand{\td}{\widetilde}

\newcommand{\mcl}{\mathcal}
\newcommand{\msf}{\mathsf}

\newcommand{\al}{\alpha}

\newcommand{\de}{\delta}

\renewcommand{\t}{^\mathsf{t}}
\newcommand{\dr}{\partial}
\newcommand{\n}{\mathbf{n}}

\newcommand{\KK}{{\frac{K(K+1)}{2}}}
\renewcommand{\b}{\mathbf{b}}

\newcommand{\Err}{{\mathsf{Err}_N}}

\begin{document}

\author[J.-C. Mourrat]{J.-C. Mourrat}
\address[J.-C. Mourrat]{DMA, Ecole normale sup\'erieure,
CNRS, PSL University, Paris, France}
\email{mourrat@dma.ens.fr}

\date{\today}

\title[Hamilton-Jacobi equations for finite-rank matrix inference]{Hamilton-Jacobi equations for finite-rank \\ matrix inference}

\begin{abstract}
We compute the large-scale limit of the free energy associated with the problem of inference of a finite-rank matrix. The method follows the principle put forward in \cite{HJinfer} which consists in identifying a suitable Hamilton-Jacobi equation satisfied by the limit free energy. We simplify the approach of \cite{HJinfer} using a notion of weak solution of the Hamilton-Jacobi equation which is more convenient to work with and is applicable whenever the non-linearity in the equation is convex.
\end{abstract}

\maketitle

%
%
%
%%%%%%%%%%%%%%%%%%%%%%%%%%%%
%%%%%%%%%%%%%%%%%%%%%%%%%%%%
%
%
%

\section{Introduction}

We fix an integer $K \in \{1,2,\ldots\}$ once and for all, and let $(\bar x_{1,k})_{1 \le k \le K}, \ldots, (\bar x_{N,k})_{1 \le k \le K}$ be $N$ independent and identically distributed random vectors taking values in $\R^K$. We denote the law of one of these vectors by $P$, and use the shorthand notation $P_N := P^{\otimes N}$ to denote their joint law. 
In the inference problem we consider, we observe the matrix
\begin{equation}  
\label{e.def.Y}
Y := \sqrt{\frac{t}{N}} \ \bar x \, \bar x\t + W \quad \in \R^{N\times N},
\end{equation}
where $t > 0$, $W = (W_{ij})_{1 \le i,j \le N}$ is an $N$-by-$N$ matrix of independent standard Gaussians, and $\bar x\t$ denotes the transpose of $\bar x \in \R^{N\times K}$. The matrix $W$ should be thought of as noise that perturbs the observation, and we aim to recover information about the rank-$K$ matrix $\bar x \, \bar x\t$ given the observation of $Y$. 

\smallskip

In order to understand this problem, it is of particular interest to study the conditional law of $\bar x$ given $Y$. This conditional law is the Gibbs measure associated with the quantity
\begin{equation}  
\label{e.def.HN}
H_N^\circ(t,x) := \sqrt{\frac t N} x \cdot W x  + \frac t N |x\t  \bar x|^2   - \frac{t}{2N} |x\t x|^2,
\end{equation}
where $x \in \R^{N\times K}$, and where for any two matrices $A$ and $B$ of the same size, we write 
\begin{equation}
\label{e.def.matrix.norm}
A \cdot B = \tr \Ll( A\t B \Rr) \quad \text{and} \quad |A| = (A \cdot A)^\frac 1 2 = \sqrt{\tr(A\t A)} .
\end{equation}
Denoting by $\P$ the joint law of $\bar x$ and $W$, this means that for any bounded measurable function $f : \R^{N\times K} \to \R$, we have
\begin{equation}  
\label{e.rel.H}
\E \Ll[ f(\bar x) \ | \ Y \Rr] =  \frac{\int_{\R^{N\times K}} f(x) \exp \Ll( H_N^\circ(t,x) \Rr) \d P_N(x)}{\int_{\R^{N\times K}}\exp \Ll( H_N^\circ(t,x) \Rr) \d P_N(x)}.
\end{equation}
As in problems of statistical mechanics, it is highly informative to understand the large-$N$ limit of the denominator in the expression above, which we may call the ``partition function''. Indeed, this quantity is essentially a moment-generating function. We aim to tackle this problem by proceeding in three steps: (1) we enrich the ``energy'' in \eqref{e.def.HN} by adding a simpler term where the quadratic interaction term $x\cdot W x$ is replaced by a linear term; (2) we find a relationship between the derivatives of the logarithm of the enriched partition function, up to error terms; (3)~we show that the effect of the error terms becomes negligible in the large-$N$ limit. We refer to the discussion of the Curie-Weiss model in \cite[Section~1]{HJinfer} for a more concrete illustration of this plan, and for further motivations.

\smallskip

As we enrich the energy in \eqref{e.def.HN}, it will be of fundamental importance for the analysis of the problem that we preserve the inference structure evidenced in \eqref{e.rel.H}. We thus define the enriched model indirectly by considering that we observe, in addition to $Y$ in \eqref{e.def.Y}, the quantity
\begin{equation}  
\label{e.obs.2}
Y' := \bar x \sqrt{h} + z \quad \in \R^{N\times K},
\end{equation}
where $z = (z_{i,k})_{1 \le i \le N, 1 \le k \le K}$ is an $N$-by-$K$ matrix of independent standard Gaussian entries, independent of $(\bar x,W)$, and $h$ is a fixed $K$-by-$K$ symmetric positive semidefinite matrix. The conditional law of $\bar x$ given the observation of $Y$ and $Y'$ is the Gibbs measure associated with the quantity defined, for every $x \in \R^{N\times K}$, by
\begin{equation*}  %\label{e.}
H_N(t,h,x) := H_N^\circ(t,x) +  \sqrt{h} \cdot x\t z + h\cdot x\t\bar x - \frac 1 2  h \cdot x\t x  .
\end{equation*}
The proof of this fact is recalled in the appendix. As explained above, our goal is to study the large-$N$ limit of the ``free energy''
\begin{equation}
\label{e.def.FN}
F_N(t,{h}) := \frac 1 N \log\Ll( \int_{\R^{N\times K}} e^{ H_N(t,{h},x) } \, \d P_N(x) \Rr) ,
\end{equation}
or of its expectation (with respect to the variables $\bar x$, $W$ and $z$)
\begin{equation}  
\label{e.def.barFN}
\bar F_N(t,{h}) := \E \Ll[ F_N(t,h) \Rr] .
\end{equation}
To state the main result, we introduce some definitions. First, notice that $\bar F_N(0,h)$ does not depend on $N$; we denote it by
\begin{equation*}  %\label{e.}
\psi(h) := \bar F_N(0,h).
\end{equation*}
We denote the set of $K$-by-$K$ symmetric matrices by $S^K$, and write $S^K_{+}$ and $S^K_{++}$ for the subsets of positive semidefinite and positive definite matrices respectively. For any open set $U \subset S^K$ and smooth function $f : U \to \R$, we define the gradient of~$f$, which we denote by $\nabla f :U \to S^K$, to be the unique mapping such that, for every $h \in U$ and $a \in S^K$, we have
\begin{equation*}  %\label{e.}
\lim_{\substack{\ep \to 0}} \ep^{-1} \Ll( f(h + \ep a) - f(h) \Rr) = \nabla f(h) \cdot a.
\end{equation*}
This gradient has a formal adjoint in $L^2(S^K)$, which we denote by $-\nabla \cdot$, and we set $\Delta := \nabla \cdot \nabla$ to denote the corresponding Laplacian. The set $S^K$ can be identified with~$\R^\KK$ after the choice of an orthonormal basis, and these differential operators on $S^K$ then match the standard definitions on $\R^\KK$.
We will most of the time encounter functions defined on $\R_+ \times S^K_+$, in which case the operators $\nabla$ and $\Delta$ are understood to act on the second variable only, keeping the first variable fixed, which we interpret as a ``time'' variable. Here is the main result of the paper.
%\begin{equation*}  %\label{e.}
%\msf H(\nabla f) := \sum_{1 \le k < l \le K} \Ll( \dr_{h_{kl}}f -  \dr_{h_{kk}}f  -  \dr_{h_{ll}}f \Rr) ^2 + 2\sum_{k = 1}^K \Ll( \dr_{h_{kk}} f \Rr) ^2.
%\end{equation*}
\begin{theorem}[convergence to HJ]
Let $f$ be the unique weak solution $f : \R_+ \times S^K_+ \to \R$ of the Hamilton-Jacobi equation 
\label{t.hj}
\begin{equation}  
\label{e.hj}
 \partial_t f - 2|\nabla f|^2   = 0 \quad  \text{in } \R_+ \times S^K_+
\end{equation}
with initial condition $f(0,h) = \psi(h)$. For every $M \ge 1$, there exists a constant $C < \infty$ such that for every $N \ge 2$ and $t \in [0,M]$, 
\begin{equation}  
\label{e.convergence}
\int_{|h| \le M} \Ll| \bar F_N -f \Rr|(t,h) \, \d h \le \frac {C}{\log N}.
\end{equation}
\end{theorem}
In \eqref{e.convergence}, the variable of integration is implicitly understood to range in $S^K_+$; the notation $\d h$ stands for the $\KK$-dimensional Lebesgue measure on this set. The convergence of $\bar F_N$ to $f$ in \eqref{e.convergence} can easily be improved to, say, convergence in~$L^{\infty}_{\mathrm{loc}}(\R_+\times S^K_+)$, using that the functions $\bar F_N$ are Lipschitz uniformly over $N$. Using the local semiconvexity of $\bar F_N$, see Definition~\ref{def.weaksol} and \eqref{e.conv.F} below, one can also obtain the convergence of the derivatives of~$\bar F_N$ to those of $f$ at every point of differentiability of~$f$. In particular, calculating the large-$N$ limit of $\partial_t \bar F_N(t,0)$ allows to identify the asymptotic minimum mean-square error of the original inference problem, see \cite{lm}.

\smallskip

I do not know if the rate of convergence in \eqref{e.convergence} is sharp. In the special rank-one case $K = 1$, the proof given below simplifies in several ways, most importantly in relation with Remark~\ref{r.convex} below, and yields an algebraic instead of logarithmic rate of convergence. 

\smallskip

As the proof reveals, the constant $C$ in \eqref{e.convergence} can be chosen to be a power of the rank~$K$. The result thus allows to let $K$ diverge slowly with $N$. Also, the independence assumption on the raws of $\bar x$ can be relaxed: what we really need is that $\bar F_N(0,\cdot)$ converges in $L^1_{\mathrm{loc}}(S^K_+)$. Finally, the Gaussian assumption on the noise~$W$ can be relaxed as well using \cite[Lemma~4]{ch06}.

\smallskip

Although it appears there in a different formulation, the qualitative convergence of $\bar F_N$ to $f$ was already proved in \cite{lm}. Besides the fact that Theorem~\ref{t.hj} gives a quantitative estimate, the main contribution of the present paper is to provide an alternative proof of this result, which I view as simpler and more ``conceptual'' than the original proof. The argument is simple enough that obtaining a rate of convergence essentially comes without additional effort. The driving idea is similar to that in~\cite{HJinfer}, in that we first show that $\bar F_N$ satisfies the Hamilton-Jacobi equation~\eqref{e.hj} approximately, and then pass to the limit. However, compared with~\cite{HJinfer}, one important difference is that we use here a notion of weak solution of the Hamilton-Jacobi equation which differs from the notion of viscosity solution used in~\cite{HJinfer}. The precise definition of weak solution we rely on is given below in Section~\ref{s.weak}. We also explain there why this notion of solution is more adapted to the purpose of proving Theorem~\ref{t.hj}. 

\smallskip

The primary ingredient for proving Theorem~\ref{t.hj} is the following result. We denote the condition number of a matrix $h \in S^K_+$ by
\begin{equation}
\label{e.def.kappa}
\kappa(h) := 
\Ll|
\begin{array}{ll}
|h| \, \Ll|h^{-1}\Rr| & \text{if } h \in S^K_{++},
\\
+\infty & \text{otherwise}.
\end{array}
\Rr.
\end{equation}
\begin{proposition}[approximate HJ in finite volume]
There exists $C < \infty$ (which depends only on $K$ and on bounds on the support of $P$) such that for every $N \ge 1$ and uniformly over $\R_+ \times S^K_+$,
\label{p.approx.hj}
\begin{equation}  
\label{e.approx.hj}
0 \le \partial_t \bar F_N - 2|\nabla \bar F_N|^2 \le C \kappa(h) N^{ - \frac 1 4 } \Ll( \Delta \bar F_N + C \Ll| h^{-1} \Rr|  \Rr) ^\frac 1 4 +C \E \Ll[ \Ll| \nabla F_N - \nabla \bar F_N \Rr| ^2 \Rr].
\end{equation}
\end{proposition}
In \eqref{e.approx.hj}, we understand that the notation $\Ll|h^{-1}\Rr|$ stands for the mapping $(t,h) \mapsto \Ll|h^{-1}\Rr|$, and that the right side is infinite whenever $h \in S^K_+ \setminus S^K_{++}$.
%In particular, the right side of \eqref{e.approx.hj} is infinite on $\R_+ \times \dr S^K_+$. 
In \eqref{e.approx.hj} and throughout the paper, whenever a statement of the form ``$\mcl P(x^\alpha)$ holds'' appears, with $\alpha \in (0,1)$, the statement should be understood as ``$x \ge 0$ and $\mcl P(x^\alpha)$ holds''. For instance, the fact that $\Delta \bar F_N + C \Ll| h^{-1} \Rr|  \ge 0$ is implied by \eqref{e.approx.hj}.

\smallskip

In order to prove Proposition~\ref{p.approx.hj}, one needs to find an upper bound on the variance of the $K$-by-$K$ matrix~$x\t \bar x$, see Lemma~\ref{l.first.approx.hj} below. Compared with the rank-one case, an additional difficulty appears, since the generalization of the rank-one argument only gives information about the symmetric part of the matrix~$x\t \bar x$. For general spin glass problems, this difficulty was resolved in \cite{pan.potts,pan.vec} using relatively involved combinatorial arguments of Ramsey type. In the simpler setting of inference problems, a more direct approach was discovered in \cite{bar19}, and we will essentially follow the same line of reasoning here.

\smallskip

Previous works on the problem, be it rank-one or more general, include \cite{dm,lkz,bar1,lm, bm1,bmm,eak}. We refer to \cite{lm} for a precise descprition of these results. 

\smallskip

The rest of the paper is organized as follows. In Section~\ref{s.weak}, we define the notion of weak solution of \eqref{e.hj} and show well-posedness of this equation. We next prove Proposition~\ref{p.approx.hj} in Section~\ref{s.approx.hj}, and Theorem~\ref{t.hj} in Section~\ref{s.convergence}. In order to make the paper self-contained, we provide a proof of \eqref{e.rel.H} and its generalization to the enriched model in an appendix.

%\begin{corollary}[convergence of derivatives]
%\label{c.deriv}
%(1) Fix $h \in S^K_+$, and set
%\begin{equation*}  %\label{e.}
%\mcl B_h := \Ll\{ t > 0 \ : \ t'\mapsto f(t',h) \mbox{ is not differentiable at $t$} \Rr\}.
%\end{equation*}
%The set $\mcl B_h$ is countable, and moreover, for every $t \in (0,\infty) \setminus \mcl B_h$, we have
%\begin{equation*}  %\label{e.}
%\lim_{N\to \infty} \dr_t \bar F_N(t,h) = \dr_t f(t,h).
%\end{equation*}

%(2) Fix $t \ge 0$, and set
%\begin{equation*}  %\label{e.}
%\mcl B_t' := \Ll\{ h \in S^K_{++} \ : \ h' \mapsto f(t,h') \mbox{ is not differentiable at $h$} \Rr\}.
%\end{equation*}
%The set $\mcl B_t'$ has null Lebesgue measure, and moreover, for every $h \in S^K_{++} \setminus \mcl B_t'$, we have
%\begin{equation*}  %\label{e.}
%\lim_{N \to \infty} \nabla \bar F_N(t,h) = \nabla f(t,h).
%\end{equation*}
%\end{corollary}

%
%
%
%%%%%%%%%%%%%%%%%%%%%%%%%%%%
%%%%%%%%%%%%%%%%%%%%%%%%%%%%
%
%
%

\section{Weak solutions of Hamiton-Jacobi equations}
\label{s.weak}

In this section, we define precisely the notion of weak solution appearing in Theorem~\ref{t.hj}, and prove the well-posedness of the Hamilton-Jacobi equation \eqref{e.hj}. We also discuss, in relation with the specific features of our problem, the advantages of this notion compared with that of viscosity solution. 

\subsection{Definition of weak solution}

In order to make the structure of the equation more salient, we give ourselves a function $\msf H \in C(S^K_+; \R)$, and consider equations of the form
\begin{equation}
\label{e.gen.hj}
\dr_t f - \msf H(\nabla f) = 0 \qquad \text{in } \R_+ \times S^K_+.
\end{equation} 
We will always assume that the function $\msf H$ is \emph{convex}. In view of the statement of Theorem~\ref{t.hj}, we are mostly interested in the case when $\msf H(p) = 2|p|^2$. In order to state the definition of weak solution of \eqref{e.gen.hj}, we introduce, for every $\de >0$, 
\begin{equation}
\label{e.def.Skde}
S^K_{+\de} := \de \, \mathrm{I}_K + S^K_+,
\end{equation}
where $\mathrm{I}_K$ denotes the $K$-by-$K$ identity matrix. In words, the set $S^K_{+\de}$ is the set of symmetric matrices with spectrum in $[\de,+\infty)$. 
\begin{definition}  
\label{def.weaksol}
We say that a Lipschitz function $f : \R_+ \times S^K_+\to \R$ is a \emph{weak solution} of \eqref{e.gen.hj} if the following conditions hold:
\begin{itemize}  %\label{}
\item
the relation \eqref{e.gen.hj} holds almost everywhere in $\R_+ \times S^K_+$;
\item 
For every $t \ge 0$, 
the mapping $h \mapsto f(t,h)$ is nondecreasing.
\item
For every $\de \in (0,1]$, there exists $C_\de < \infty$ such that for every $t \in [\de,\de^{-1}]$, the mapping $h \mapsto f(t,h) + C_\de|h|^2$ is convex on the set $\{h \in S^K_{+\de} \ : \ |h|\le \de^{-1}\}$.
\end{itemize}
\end{definition}
We refer to the second and third conditions in Definition~\ref{def.weaksol} as the \emph{monotonicity} and \emph{local semiconvexity} conditions, respectively. Let us clarify the meaning of the monotonicity condition. For every $A,B \in S^K$, we write $A \le B$ if and only if $B-A \in S^K_+$. This defines a partial order on~$S^K$. We say that a function $g : S^K_+ \to \R$ is nondecreasing if, for every $h,h' \in S^K_+$, we have
\begin{equation}  
\label{e.monotone}
h \le h' \quad \implies \quad \msf H(h) \le \msf H(h').
\end{equation}

\smallskip

Recall that by the Rademacher theorem, a Lipschitz function is differentiable almost everywhere. In view of the second part of the following elementary lemma, if $f$ is Lipschitz and satisfies the monotonicity condition, it then makes sense to ask about the measure of the set of points where \eqref{e.gen.hj} holds. 

\begin{lemma}
\label{l.grad}
(1) Let $a \in S^K$. We have
\begin{equation}
\label{e.equiv.pos}
a \in S^K_+ \quad \iff \quad \forall b \in S^K_+, \ a \cdot b \ge 0.
\end{equation}

(2) Let $f : S^K_+ \to \R$ be a Lipschitz function. The function $f$ is nondecreasing if and only if, for almost every $a \in S^K_+$, we have $\nabla f(a) \in S^K_+$.
\end{lemma}
\begin{proof}
If $a, b \in S^K_+$, then $a\cdot b = |\sqrt{a}\sqrt{b} |^2 \ge 0$, and thus the direct implication in~\eqref{e.equiv.pos} holds. For the converse implication, without loss of generality, we can assume that $a$ is a diagonal matrix, in which case the result is easily derived by considering diagonal matrices for $b$. For part (2), by approximation (see e.g.\ \eqref{e.def.f.eps} and \eqref{e.ae.lim} below), we can assume that the function $f$ is smooth. For every $\ep > 0$ and $a,b \in S^K_+$, we have
\begin{equation*}  %\label{e.}
0 \le \ep^{-1}(f(a + \ep b) - f(a)),
\end{equation*}
and the right side tends to $\nabla f(a)\cdot b$ as $\ep$ tends to $0$. 
Using also \eqref{e.equiv.pos}, we obtain the direct implication in part (2). The converse statement follows by writing, for every $a \le b \in S^K_+$,
\begin{equation*}  %\label{e.}
f(b) - f(a) = \int_0^1 (b-a) \cdot \nabla f(a + s(b-a)) \, \d s,
\end{equation*}
and using again \eqref{e.equiv.pos} to conclude.
\end{proof}

The monotonicity condition in Definition~\ref{def.weaksol} is only really used in a neighborhood of the set $\R_+\times \Ll(S^K_{+} \setminus S^K_{++}\Rr)$, and plays the role of a one-sided boundary condition of Neumann type. We will combine this with the additional assumption that $\msf H$ is nondecreasing to obtain the uniqueness of solutions. In the case of domains without boundary, say $h \in S^K$, then both conditions can be dropped (assuming then that we are given a convex $\msf H$ defined on the entirety of $S^K$, not just on $S^K_+$). Moreover, as will be seen in the proof of uniqueness of solutions given below, these conditions can be weakened significantly. Roughly speaking, we need that $\nabla \msf H(\nabla f) \cdot \mathbf{n}_{S^K_+} \le 0$ almost everywhere on $\dr S^K_+$, where $\mathbf{n}_{S^K_+}$ is the unit outer normal to~$S^K_+$ (seen as a subset of $S^K$ with Lipschitz boundary).

\smallskip

While this monotonicity condition on weak solutions can be weakened significantly, it cannot be dropped altogether without loosing the uniqueness of weak solutions. For an example with $K = 1$, for any $p \in \R_+$, the function $(t,h) \mapsto (\msf H(-p) t - ph)_+$ satisfies \eqref{e.gen.hj} almost everywhere, is convex in $h$, and is constant equal to $0$ at $t = 0$. Similarly, the local semiconvexity condition cannot be dropped without loosing the uniqueness of weak solutions. 

\subsection{Well-posedness and comparison with viscosity-solution approach}
The next proposition shows that the notion of weak solution introduced in Definition~\ref{def.weaksol} indeed ensures the uniqueness of solutions. The proof essentially follows the classical approach of \cite{dou65,kru66,kru67}, see also \cite{benton} and \cite[Theorem~I.3.3.7]{evans}, that we adapt to our particular setting (in particular in relation with the boundary condition).

\begin{proposition}[uniqueness of weak solutions]
\label{p.uniqueness}
Let $\msf H \in C^2(S^K_+,\R)$ be convex and nondecreasing. If $f$ and $g$ are two weak solutions to \eqref{e.gen.hj} such that $f(0,\cdot) = g(0,\cdot)$, then $f = g$.
\end{proposition}
\begin{proof}
We decompose the proof into three steps. 

\smallskip

\emph{Step 1.} We identify an equation satisfied by the difference $w := f - g$. The following identities hold almost everywhere in $\R_+\times S^K_+$:
\begin{align*}  %\label{e.}
\dr_t w 
& = \dr_t f - \dr_t g
\\
& = 
\msf H(\nabla f) - \msf H(\nabla g)
\\
& = 
\int_0^1 \dr_a \Ll( \msf H(a\nabla f + (1-a)\nabla g) \Rr) \, \d a
\\
& = \int_0^1 \nabla w \cdot \nabla \msf H(a\nabla f + (1-a)\nabla g)  \, \d a.
\end{align*}
Setting 
\begin{equation*}  %\label{e.}
\b := \int_0^1 \nabla \msf H(a\nabla f + (1-a)\nabla g)  \, \d a,
\end{equation*}
we thus have that 
\begin{equation}  
\label{e.eq.for.w}
\dr_t w - \b \cdot \nabla w = 0 \quad \text{a.e. in } \R_+\times S^K_+.
\end{equation}
Let $\phi \in C^\infty(\R)$ be a nonnegative smooth function that will be specified in the course of the argument. 
We set $v := \phi(w)$, and multiply \eqref{e.eq.for.w} by $\phi'(w)$ to obtain that
\begin{equation*}  %\label{e.}
\dr_t v - \b \cdot \nabla v = 0 \quad \text{a.e. in } \R_+\times S^K_+.
\end{equation*}

\smallskip

\emph{Step 2.} Roughly speaking, the idea of the proof is to observe that the ``local mass'' of $v$ cannot increase much, using integration by parts and the fact that $\nabla \cdot \b$ is bounded below. Since $\nabla \cdot \b$ is not well-defined pointwise, we first regularize $f$ and $g$.

\smallskip

Let $\zeta \in C^\infty_c(S^K)$ be a smooth function with compact support in $\{h \in S^K \ : \ |h| \le 1\}$ and such that $\int_{S^K} \zeta = 1$. For every $\ep > 0$, we set
\begin{equation}  
\label{e.def.zetaeps}
\zeta_\ep := \ep^{-\KK} \zeta\Ll(\frac{\cdot}{\ep}\Rr).
\end{equation}
We define the mollified functions on $\R_+ \times S^K_{+\ep}$
\begin{equation}  
\label{e.def.f.eps}
f_\ep := f \ast \zeta_\ep \quad \text{ and } \quad g_\ep := g \ast \zeta_\ep,
\end{equation}
where $\ast$ denotes convolution in the $h$ variable only. Explicitly, for every $t \ge 0$ and $h \in S^K_{+\ep}$,
\begin{equation*}  %\label{e.}
f_\ep(t,h) := \int_{S^K} f(t,h-h') \zeta_\ep(h') \, \d h',
\end{equation*}
where on the right side, the notation $\d h'$ stands for the $\KK$-dimensional Lebesgue measure on $S^K$. 
The definitions of $f_\ep$ and $g_\ep$ make sense since whenever $|h'| \le \ep$ and $h \in S^K_{+\ep}$, we have $\ep \mathrm{I}_K - h' \in S^K_+$ and thus $h-h' \in S^K_{+}$. Moreover, in the case when $h \in S^K_{+\de}$ for some $\de \in[ 2 \ep,1]$, we have $h-h' \in S^K_{+\frac \de 2}$. By the local semiconvexity assumption on $f$ and $g$, there exists a constant $C_\de < \infty$ such that for every $\ep \le \frac \de 2$ and $t \in [\de,\de^{-1}]$,
\begin{equation}  
\label{e.convex.fep}
\text{the mapping }
\Ll\{
\begin{array}{rcl}  %\label{}
\{h \in S^K_{+\de} \ : \ |h| \le \de^{-1}\} & \to & \R \\
h & \mapsto & f_\ep(t,h) + C_\de|h|^2
\end{array}
\Rr.
\text{ is convex},
\end{equation}
and the same property holds with $f_\ep$ replaced by $g_\ep$. We also have that 
\begin{equation}  
\label{e.ae.lim}
\lim_{\ep \to 0} \nabla f_\ep = \nabla f, \qquad \lim_{\ep \to 0} \nabla g_\ep = \nabla g \quad \text{a.e. in } \R_+ \times S^K_+
\end{equation}
(see e.g.\ \cite[Theorem~C.7]{evans} for a proof), as well as
\begin{equation}  
\label{e.grad.bound}
\|\nabla f_\ep\|_{L^\infty(\R_+ \times S^K_{+\ep})} \le \|\nabla f\|_{L^\infty(\R_+ \times S^K_{+})}, \qquad  \|\nabla g_\ep\|_{L^\infty(\R_+ \times S^K_{+\ep})} \le \|\nabla g\|_{L^\infty(\R_+ \times S^K_{+})}.
\end{equation}
For future reference, we introduce the shorthand notation
\begin{equation}  
\label{e.def.L}
L := \max \Ll( \|\nabla f\|_{L^\infty(\R_+ \times S^K_{+})}, \|\nabla g\|_{L^\infty(\R_+ \times S^K_{+})} \Rr) .
\end{equation}
Throughout the proof, we enforce without further mention that $\de > 0$ is sufficiently small that $\de^{-1} \ge L$. 
Identifying $(S^K,|\cdot|)$ with the Euclidean space $\R^\KK$, we denote the Hessian of a function $\td f$ defined on an open subset of $S^K$ by $\nabla^2 \td f$. For the purposes of this proof, we think of $\nabla^2 \td f$ as being a $\KK$-by-$\KK$ symmetric matrix. By \eqref{e.convex.fep}, we have, for every $\ep \le \frac \de 2$, $t \in [\de,\de^{-1}]$ and $h \in S^K_{+\de}$ with $|h|\le \de^{-1}$, 
\begin{equation}  
\label{e.explicit.semiconv}
\nabla^2 f_\ep(t,h) + 2C_\de \mathrm{I}_\KK \in S^\KK_+ \quad \text{and} \quad \nabla^2 g_\ep(t,h) + 2C_\de \mathrm{I}_\KK \in S^\KK_+.
\end{equation}
We set
\begin{equation*}  %\label{e.}
\b_\ep := \int_0^1 \nabla \msf H(a\nabla f_\ep + (1-a)\nabla g_\ep)  \, \d a,
\end{equation*}
and observe that
\begin{equation*}  %\label{e.}
\dr_t v - \b_\ep \cdot \nabla v = (\b - \b_\ep)\cdot \nabla v \quad \text{a.e. in } \R_+\times S^K_{+\ep},
\end{equation*}
that is,
\begin{equation}  
\label{e.eq.v.bep}
\dr_t v - \nabla \cdot (\b_\ep v) + v \nabla \cdot \b_\ep = (\b - \b_\ep )\cdot \nabla v \quad \text{a.e. in } \R_+\times S^K_{+\ep}.
\end{equation}
Moreover,
\begin{equation*}  %\label{e.}
\nabla \cdot \b_\ep = \int_0^1 (a \nabla^2 f_\ep + (1-a) \nabla^2 g_\ep) \cdot \nabla^2 \msf H(a\nabla f_\ep + (1-a)\nabla g_\ep) \, \d a,
\end{equation*}
where, as for $\nabla^2 f_\ep$ and $\nabla^2 g_\ep$, we think of $\nabla^2 \msf H$ as being a $\KK$-by-$\KK$ symmetric matrix. Since $\msf H$ is convex, the $\KK$-by-$\KK$ matrix $\nabla^2 \msf H(p)$ is positive semidefinite for every $p \in S^K_+$. It thus follows from \eqref{e.equiv.pos} and \eqref{e.explicit.semiconv} that  for every $\ep \le \frac \de 2$, we have
\begin{equation*}  %\label{e.}
\int_0^1 (a \nabla^2 f_\ep + (1-a) \nabla^2 g_\ep + 2C_\de \mathrm{I}_{\KK}) \cdot \nabla^2 \msf H(a\nabla f_\ep + (1-a)\nabla g_\ep) \, \d a \ge 0 \quad \text{on } [\de,\de^{-1}] \times S^K_{+\de}.
\end{equation*}
Hence, by \eqref{e.grad.bound} and \eqref{e.def.L}, we get that for every $\ep \le \frac \de 2$,
\begin{equation*}
%\label{e.nonneg.bep}
\nabla \cdot \b_\ep \ge -C_\de \sup \Ll\{ |\nabla^2 \msf H(p)| \ : \ |p| \le L \Rr\}   \qquad \text{on } \ [\de, \de^{-1}]\times S^K_{+\de}.
\end{equation*}
Up to a redefinition of $C_\de < \infty$, we may thus assume that for every $\ep \le \frac \de 2$,
\begin{equation}
\label{e.nonneg.bep}
\nabla \cdot \b_\ep  +C_\de \ge 0   \qquad \text{on } \ [\de, \de^{-1}]\times S^K_{+\de}.
\end{equation}

\smallskip

\emph{Step 3.} We are now ready to implement the argument announced at the beginning of Step 2. 
Denote
\begin{equation*}  %\label{e.}
R := 1+\sup \Ll\{ |\nabla \msf H(p)| \ : \ |p| \le L \Rr\} .
\end{equation*}
We fix $T \ge 1$ and define, for every $t \in [0,\frac T 2]$,
\begin{equation}  
\label{e.def.Bde}
B_\de(t) := \Ll\{ h \in S^K_{+\de} \ : \ |h| \le R(T-t) \Rr\},
\end{equation}
\begin{equation}  
\label{e.def.b+Bde}
\dr_{+} B_\de(t) := \Ll\{ h \in S^K_{+\de} \ : \ |h| = R(T-t) \Rr\} ,
\end{equation}
and 
\begin{equation}  
\label{e.def.bBde}
\dr_0 B_\de(t) := \Ll\{ h \in \dr S^K_{+\de}  \ : \ |h| \le R(T-t) \Rr\} .
\end{equation}
We assume without further mention that $\de > 0$ is sufficiently small that $\de^{-1} \ge RT$. 
Up to a set of null $\frac{K(K-1)}{2}$-Hausdorff measure, the boundary of $B_\de(t)$ is the disjoint union of $\dr_{+} B_\de(t)$ and $\dr_0 B_\de(t)$. We aim to obtain a Gronwall inequality for the quantity
\begin{equation}  
\label{e.def.Jde}
J_\de(t) := \int_{B_\de(t)} v(t,\cdot) = \int_{B_\de(t)} v(t,h) \, \d h \qquad (t \in [0,\tfrac T 2]).
\end{equation}
The function $J_\de$ is Lipschitz, and for almost every $t \in [0,\tfrac T 2]$, we have
\begin{equation*}  %\label{e.}
\dr_t J_\de(t) = \int_{B_\de(t)} \dr_t v(t,\cdot) - R \int_{\dr_+ B_\de(t)} v(t,\cdot),
\end{equation*}
where the second integral is a boundary integral (with respect to the $\frac{K(K-1)}{2}$-dimensional Hausdorff measure on $\dr_+B_\de(t)$). Using \eqref{e.eq.v.bep}, \eqref{e.nonneg.bep}, and that $v \ge 0$, we get that for almost every $t \in [\de,\tfrac T 2]$,
\begin{align*}  %\label{e.}
\dr_t J_\de(t) 
& = \int_{B_\de(t)} \Ll((\b - \b_\ep) \cdot \nabla v -  v \nabla \cdot \b_\ep + \nabla \cdot(\b_\ep v)\Rr)(t,\cdot)- R \int_{\dr_+ B_\de(t)} v(t,\cdot)
\\
& \le C_\de J_\de(t) + \int_{B_\de(t)} \Ll((\b - \b_\ep) \cdot \nabla v\Rr)(t,\cdot) 
\\
& \qquad \qquad + \int_{\dr_+B_\de(t)} \Ll( (\b_\ep \cdot \n - R ) v\Rr)(t,\cdot) + \int_{\dr_0 B_\de(t)} \Ll(v \b_\ep \cdot \n\Rr)(t,\cdot),
\end{align*}
where we denote by $\n$ the unit outer normal to $B_\de(t)$. Using \eqref{e.ae.lim} and the dominated convergence theorem, we see that the first integral on the right side above tends to $0$ as $\ep$ tends to $0$. The first boundary integral on the right side above is nonpositive, by the definitions of $\b_\ep$ and $R$, and \eqref{e.grad.bound}. We now show that
\begin{equation}  
\label{e.dr0}
\int_{\dr_0 B_\de(t)} \Ll(v \b_\ep \cdot \n\Rr)(t,\cdot)\le 0.
\end{equation}
We first notice that $-\n \in S^K_+$ almost everywhere on $\dr_0 B_\de(t)$. By \eqref{e.equiv.pos}, in order to show \eqref{e.dr0}, it suffices to verify that $\b_\ep \in S^K_+$. This follows from the assumption that $\msf H$, $f$ and $g$ are nondecreasing and an application of Lemma~\ref{l.grad}. Summarizing, we have shown that for almost every $t \in [\de,\tfrac T 2]$,
\begin{equation}  
\label{e.gronwall}
\dr_t J_\de(t) \le C_\de J_\de(t).
\end{equation}
Finally, we observe that 
\begin{equation*}  %\label{e.}
w(\de,h) = (f-g)(\de,h) \le \de \|\dr_t (f-g)\|_{L^\infty(\R_+\times S^K_+)}.
\end{equation*}
Selecting $\phi \in C^\infty(\R)$ in such a way that, for every $x \in \R$,
\begin{equation*}  %\label{e.}
\phi(x) > 0 \quad \iff \quad |x| > \de \|\dr_t (f-g)\|_{L^\infty(\R_+\times S^K_+)}
\end{equation*}
ensures that $J_\de(\de) = 0$. Combined with \eqref{e.gronwall} and the fact that $J_\de \ge 0$, this yields that for every $t \in [\de,\frac T 2]$,
\begin{equation*}  %\label{e.}
J_\de \Ll(t\Rr)=0,
\end{equation*}
and in particular, for every $t \in [\de,\frac T 2]$ and almost every $h \in B_\de(t)$,
\begin{equation*}  %\label{e.}
|(f-g)(t,h)| \le \de \|\dr_t (f-g)\|_{L^\infty(\R_+\times S^K_+)} .
\end{equation*}
Letting $\de$ tend to $0$ yields the desired result.
\end{proof}

As long as the assumptions of monotonicity and convexity of $\msf H$ are satisfied, the notion of weak solution in Definition~\ref{def.weaksol} turns out to be much more convenient to work with than the notion of viscosity solution employed in \cite{HJinfer}, as we now explain.

\smallskip

To start with, the monotonicity condition on the solution allows to circumvent the relatively cumbersome treatment of the viscosity-solution interpretation of the Neumann boundary condition used in \cite{HJinfer}, replacing it with a straightforward verification of the fact that $h \mapsto \bar F_N(t,h)$ is nondecreasing for every $N$. 

\smallskip

More significantly, as was shown in \cite[Section~4]{HJinfer}, the proper treatment of odd-degree tensor versions of this problem no longer involve one-sided estimates of the general form $0 \le \dr_t \bar F_N - \msf H(\bar F_N) \le N^{-1} \Delta \bar F_N + \cdots$, but rather two-sided estimates of the form $|\dr_t \bar F_N - \msf H(\bar F_N)| \le N^{-1} \Delta \bar F_N + \cdots$ (whether or not the right-hand side is raised to a power $\al \in (0,1)$ is irrelevant to this discussion, so we ignore it). In principle, this is a worrisome situation, since viscosity solutions are meant to ``remember the sign of the Laplacian'' in the vanishing viscosity limit. The reason why the proof could still be successfully carried out in spite of this is precisely by leveraging on the additional information that $\bar F_N$ is locally semiconvex. Arguments based on the notion of weak solution use this property in a much more transparent way. 

\smallskip

Finally, arguments based on weak solutions are more adapted to the type of error terms that appear in the right-hand side of \eqref{e.approx.hj}. The proof of convergence of $\bar F_N$ will be a more quantitative version of the argument in Proposition~\ref{p.uniqueness}: instead of showing that two weak solutions must be equal, we will show that two ``almost weak solutions'' must be almost equal. This proof is most suited to accommodate for error terms that are estimated in an $L^\infty_t L^1_h$-type norm, see for instance the definition of~$J_\de$ in \eqref{e.def.Jde}. The error terms on the right side of \eqref{e.approx.hj}, in particular $|\nabla F_N - \nabla \bar F_N|^2$ (but also $N^{-1} \Delta \bar F_N$), can be estimated in such a norm in a straightforward way. On the other hand, the notion of viscosity solution handles most naturally errors that are estimated in $L^\infty_t L^\infty_h$. While this problem can be (and has been) circumvented by appealing to local convolution etc., the approach based on weak solutions is more straightforward and easily yields quantitative estimates.

\smallskip

To conclude this section, we give the Hopf-Lax formula for weak solutions to \eqref{e.gen.hj}, which in particular proves the existence of solutions to \eqref{e.gen.hj}. Except for the treatment of the boundary condition, the argument is classical. We denote by $\msf H^*$ the convex dual of $\msf H$: that is, for each $q \in S^K$, we set
\begin{equation}  
\label{e.def.H*}
\msf H^*(q) := \sup_{p \in S^K_+} \Ll( p\cdot q - \msf H(p) \Rr) .
\end{equation}
\begin{proposition}[Hopf-Lax formula]
\label{p.hopf.lax}
Let $H \in C(S^K_+;\R)$ be a convex function such that $H(p)$ depends only on $|p|$. 
Let $\psi : S^K_+ \to \R$ be a nondecreasing Lipschitz function satisfying the following local semiconvexity property: for every $\de > 0$, there exists a constant $C_\de < \infty$ such that 
\begin{equation}  
\label{e.local.semiconvex}
\text{the mapping }
\Ll\{
\begin{array}{rcl}  %\label{}
S^K_{+\de} & \to & \R \\
h & \mapsto & \psi(h) + C_\de \, |h|^2
\end{array}
\Rr.
\text{ is convex}.
\end{equation}
For each $t \ge 0$ and $h \in S^K_+$, we define
\begin{equation}  
\label{e.def.f}
f(t,h) := \sup_{h' \in S^K_+} \Ll( \psi(h') - t \msf H^*\Ll(\frac{h'-h}{t}\Rr) \Rr) ,
\end{equation}
with the understanding that $f(0,\cdot) = \psi$. The function $f$ is a weak solution of \eqref{e.gen.hj}. 
\end{proposition}
\begin{remark}
The local semiconvexity assumption \eqref{e.local.semiconvex} of the initial condition is used to show that the mapping $h \mapsto f(t,h)$ is locally semiconvex. This property can also be obtained by assuming instead that $\msf H$ is uniformly convex, see \cite[Lemma~3.3.4]{evans}. However, it is convenient to state Proposition~\ref{p.hopf.lax} in this way because the explicit local semiconvexity property of $f$ as stated in \eqref{e.local.semiconv.fth} below will be used later to obtain a quantitative rate of convergence in Theorem~\ref{t.hj}. The assumption that $H(p)$ depends only on $|p|$ simplifies the consideration of problems related to the presence of a boundary. It can certainly be weakened significantly, but I do not know whether it can be removed altogether.
\end{remark}
\begin{proof}[Proof of Proposition~\ref{p.hopf.lax}]
We decompose the proof into five steps.

\smallskip

\emph{Step 1.} In this first step, we use that $\psi$ is nondecreasing and the symmetry assumption on $\msf H$  to assert that, for every $t \ge 0$ and $h \in S^K_+$,
\begin{align}  
\label{e.good.side}
f(t,h) & = \sup \Ll\{ \psi(h') - t \msf H^* \Ll( \frac{h'-h}{t} \Rr) \ : \ h' \in S^K_+ \text{ s.t. } h' \ge h \Rr\} 
\\
\label{e.good.side2}
& = \sup \Ll\{ \psi(h+h') - t \msf H^* \Ll( \frac{h'}{t} \Rr)  \ : \ h' \in  S^K_+ \Rr\} .
\end{align}
For every $h \in S^K$, we write $h_+$ to denote the image of $h$ under the mapping $x \mapsto \max(x,0)$. In a basis where $h$ is diagonal, this means that we replace the negative eigenvalues by zeros. 
We first show that, for every $h \in S^K$,
\begin{equation}
\label{e.msfH.flat}
\msf H^*(h) = \msf H^*(h_+). 
\end{equation}
This statement is equivalent to
\begin{equation}  
\label{e.msfH.flat2}
\sup_{p \in S^K_+} \Ll( p\cdot h - \msf H(p) \Rr) = \sup_{p \in S^K_+} \Ll( p\cdot h_+ - \msf H(p) \Rr).
\end{equation}
Since $h \le h_+$, the statement \eqref{e.msfH.flat2} with the inequality $\le$ in place of the equality is clear by \eqref{e.equiv.pos}. Conversely, for every $p' \in S^K_+$, we may choose $p := \frac{h_+}{|h_+|}{|p'|} \in S^K_+$, so that 
\begin{equation*}  %\label{e.}
|p| = |p'| \quad \text{ and } \quad p\cdot h = |p'| \, \frac{h\cdot h_+}{|h_+|}= |p'| \, |h_+|.
\end{equation*}
Since we assume that $\msf H(p)$ depends only on $|p|$, and since $p'\cdot h_+ \le |p'| \, |h_+|$, this proves the inequality $\ge$ in \eqref{e.msfH.flat2}, and therefore \eqref{e.msfH.flat}.

\smallskip

Using \eqref{e.msfH.flat}, we get that for every $h,h' \in S^K_+$,
\begin{align*}  %\label{e.}
\psi(h') - t\msf H^* \Ll( \frac{h'-h}{t} \Rr) 
& = \psi(h') - t\msf H^* \Ll( \frac{(h'-h)_+}{t} \Rr) 
\\
& \le \psi(h + (h'-h)_+) - t\msf H^* \Ll( \frac{h+(h'-h)_+-h}{t} \Rr) ,
\end{align*}
where we also used that $h'-h \le (h'-h)_+$ and that $\psi$ is nondecreasing in the last step. The identities \eqref{e.good.side}-\eqref{e.good.side2} thus follow. From \eqref{e.good.side2}, it is clear that the mapping $h \mapsto f(t,h)$ is nondecreasing.

\smallskip

\emph{Step 2.} We prove the dynamic programming principle, that is, for every $s,t \ge 0$ and $h \in S^K_+$,
\begin{equation}
\label{e.dyn.prog}
f(t+s,h) = \sup_{h' \in S^K_+} \Ll( f(t,h') -s \msf H^* \Ll(\frac{h'-h}{s}\Rr)  \Rr) .
\end{equation}
Since $\msf H^*$ is convex, we have, for every $s,t \ge 0$ and $h,h',h'' \in S^K_+$,
\begin{equation}  
\label{e.convex.H}
\msf H^* \Ll( \frac{h'-h}{t+s} \Rr) \le \frac{t}{t+s} \msf H^* \Ll( \frac{h'-h''}{t} \Rr)  + \frac{s}{t+s} \msf H^* \Ll( \frac{h''-h}{s} \Rr) ,
\end{equation}
and thus
\begin{align}  
\label{e.ineq.dynprog1}
f(t+s,h)
& \ge \sup_{h',h'' \in S^K_+} \Ll( \psi(h') - t \msf H^* \Ll( \frac{h'-h''}{t} \Rr)  - s  \msf H^* \Ll( \frac{h''-h}{s} \Rr) \Rr) 
\\
\label{e.ineq.dynprog2}
& \ge \sup_{h'' \in S^K_+} \Ll( f(t,h'') - s  \msf H^* \Ll( \frac{h''-h}{s} \Rr) \Rr) .
\end{align}
Arguing as in Step 1 (and recalling that $h \mapsto f(t,h)$ is nondecreasing), we see that we can restrict the supremum in \eqref{e.ineq.dynprog2} to $h'' \ge h$. Moreover, for each $h'' \ge h$, we can choose $h' = h + \frac{t+s}{s}(h''-h) \in S^K_+$ and observe that in this case,
\begin{equation*}  %\label{e.}
\frac{h'-h}{t+s} = \frac{h'-h''}{t} = \frac{h''-h}{s}.
\end{equation*}
With this choice of $h'$, the inequality in \eqref{e.convex.H} is an equality, and we can thus assert that the inequalities in \eqref{e.ineq.dynprog1}-\eqref{e.ineq.dynprog2} are equalities as well.

\smallskip

\emph{Step 3.} We show that $f$ is Lipschitz continuous. It follows from \eqref{e.good.side2} that for every $h,h_1 \in S^K_+$, 
\begin{equation*}  %\label{e.}
f(t,h) \le f(t,h_1) + \|\nabla \psi\|_{L^\infty} |h-h_1|,
\end{equation*}
and by symmetry, that 
\begin{equation}  
\label{e.Lip.h}
\Ll| f(t,h) - f(t,h_1) \Rr| \le \|\nabla \psi\|_{L^\infty} |h-h_1|.
\end{equation}
This shows that $f$ is Lipschitz continuous in the $h$ variable. For the regularity in time, we recall that \eqref{e.dyn.prog} also holds with the additional restriction $h' \ge h$, and appeal to~\eqref{e.Lip.h} to write, for every $s,t\ge 0$,
\begin{equation*}  %\label{e.}
f(t,h) \le f(t+s,h) \le \sup_{h' \ge h} \Ll( f(t, h) + \|\nabla \psi\|_{L^\infty}|h'-h| - s \msf H^*\Ll( \frac{h'-h}{s} \Rr) \Rr) .
\end{equation*}
Recall the definition of $\msf H^*$ in \eqref{e.def.H*}. Testing the supremum in this definition with $p = \lambda \frac{h'-h}{|h'-h|} \in S^K_+$ for some $\lambda > 0$ to be determined, we obtain that for every $h' \ge h \in S^K_+$,
\begin{equation*}  %\label{e.}
s \msf H^* \Ll( \frac{h'-h}{s} \Rr) \ge \lambda |h'-h| - s\msf H \Ll( \lambda \frac{h'-h}{|h'-h|} \Rr) .
\end{equation*}
Selecting $\lambda = \|\nabla \psi\|_{L^\infty}$, we conclude that 
\begin{equation*}  %\label{e.}
0 \le f(t+s,h ) - f(t,h) \le s \, \sup \Ll\{\msf H(p) \ : \ p \in S^K_+, \ |p| \le \|\nabla \psi\|_{L^\infty} \Rr\}.
\end{equation*}
This shows in particular that the function $f$ is Lipschitz continuous in the $t$ variable. 

\smallskip

\emph{Step 4.} We show that for every $t \ge 0$ and  $\de > 0$, 
\begin{equation}
\label{e.local.semiconv.fth}
\text{the mapping }
\Ll\{
\begin{array}{rcl}  %\label{}
S^K_{+2\de} & \to & \R \\
h & \mapsto & f(t,h) + C_\de \, |h|^2
\end{array}
\Rr.
\text{ is convex},
\end{equation}
where $C_\de$ is the constant appearing in \eqref{e.local.semiconvex}. 
Reproducing the argument in the previous step with $\lambda = 1+\|\nabla \psi\|_{L^\infty}$, we see that the supremum in \eqref{e.good.side} is attained. Fix $t > 0$, $h \in S^K_{+2\de}$, and denote by $h' \in S^K_+$ a point realizing the supremum in \eqref{e.good.side2}:
\begin{equation*}  %\label{e.}
f(t,h) = \psi(h+h') - t \msf H^* \Ll( \frac{h'}{t} \Rr) .
\end{equation*}
for every $h'' \in S^K$ sufficiently small (in terms of $\de$) to guarantee that $h-h''$ and $h + h''$ belong to $S^K_{+\de}$, we have
\begin{equation*}  %\label{e.}
f(t,h+h'') + f(t,h-h'') - 2 f(t,h) \le \psi(h+h'+h'') + \psi(h+h'-h'') - 2 \psi(h+h'). 
\end{equation*}
(This follows by writing \eqref{e.good.side2} for $f(t,h+h'')$ and $f(t,h-h'')$ and testing the supremum at $h'$.) Recalling that $h' \in S^K_+$ and using \eqref{e.local.semiconvex}, we deduce that 
\begin{equation*}  %\label{e.}
f(t,h+h'') + f(t,h-h'') - 2f(t,h) \le 2C_\de |h''|^2. 
\end{equation*}
This proves \eqref{e.local.semiconv.fth}.

\smallskip

\emph{Step 5.} Since $f$ is Lipschitz continuous, it is differentiable almost everywhere. In this final step, we show that for every $t  >0$ and $h \in S^K_{++}$, if $(t,h)$ is a point of differentiability of $f$, then the equation \eqref{e.gen.hj} is satisfied at $(t,h)$. We fix such $(t,h)$, and for every $q \in S^K$ and $s > 0$ sufficiently small, we use \eqref{e.dyn.prog} to write
\begin{equation*}  %\label{e.}
f(t+s,h) \ge f(t,h + sq) -  s\msf H^*(q).
\end{equation*}
Letting $s$ tend to $0$, we infer that
\begin{equation}  
\label{e.one.sided.eq0}
\dr_t f(t,h) - q \cdot \nabla f(t,h) +  \msf H^*(q) \ge 0.
\end{equation}
We may view the function $\msf H$ as defined on $S^K$, by setting $\msf H(p) = +\infty$ whenever $p \notin S^K_+$. This function is convex and lower semicontinuous, and thus, for every $p \in S^K$,
\begin{equation}  
\label{e.convex.duality}
\msf H^{**}(p) := \sup_{q \in S^K} \Ll(p\cdot q - \msf H^*(q)\Rr) = \msf H(p).
\end{equation}
Taking the infimum over $q$ in \eqref{e.one.sided.eq0}, we thus conclude that
\begin{equation}
\label{e.one.sided.eq}
\Ll(\dr_t f - \msf H(\nabla f)\Rr)(t,h) \ge 0.
\end{equation}
(Since $f$ is nondecreasing, we already knew by Lemma~\ref{l.grad} that $\nabla f(t,h) \in S^K_+$.)
There remains to show the converse inequality to \eqref{e.one.sided.eq}. Let $h' \in S^K_+$ be such that
\begin{equation*}  %\label{e.}
f(t,h) = \psi(h') - t \msf H^* \Ll( \frac{h'-h}{t} \Rr) .
\end{equation*}
For every $s > 0$ sufficiently small, we have
\begin{align*}  %\label{e.}
f\Ll(t-s,h +\frac s t (h'-h)\Rr) 
& \ge \psi(h') - (t-s) \msf H^* \Ll( \frac{h'-h- \frac s t (h'-h)}{t-s} \Rr) 
\\
& \ge f(t,h) + s \msf H^* \Ll( \frac{h'-h}{t} \Rr).
\end{align*}
Letting $s$ tend to $0$, we obtain that
\begin{equation*}  %\label{e.}
\dr_t f(t,h) - \frac{h'-h}{t} \cdot \nabla f(t,h) + \msf H^*\Ll( \frac{h'-h}{t} \Rr) \le 0.
\end{equation*}
Together with \eqref{e.convex.duality}, this yields the converse inequality to \eqref{e.one.sided.eq} and thus completes the proof.
\end{proof}

%
%
%
%%%%%%%%%%%%%%%%%%%%%%%%%%%%
%%%%%%%%%%%%%%%%%%%%%%%%%%%%
%
%
%

\section{Approximate Hamilton-Jacobi equation and basic estimates}
\label{s.approx.hj}
The main goal of this section is to prove Proposition~\ref{p.approx.hj}. 
We will also record basic derivative and concentration estimates that will be useful in the next section. 

\subsection{Nishimori identity}
We start by presenting the Nishimori identity, a simple but crucial property which allows to avoid facing a never-ending cascade of new replicas as we differentiate the free energy. 
We denote by $\la \cdot \ra$ the Gibbs measure associated with the energy~$H_N(t,{h},\cdot)$. That is, for each bounded measurable function $f : \R^{N\times K} \to \R$, we set
\begin{equation}  
\label{e.def.Gibbs}
\la f(x) \ra := \frac{1}{Z_N(t,{h})} \int_{\R^{N\times K}} f(x) e^{H_N(t,{h},x)} \, \d P_N(x),
\end{equation}
where 
\begin{equation*}  %\label{e.}
Z_N(t,{h}) := \int_{\R^{N\times K}} e^{H_N(t,{h},x)} \, \d P_N(x).
\end{equation*}
Note that although the notation does not display it, this random probability measure depends on $t, {h}$, as well as on the realization of the random variables~$\bar x$, $W$ and $z$. 
We also consider ``replicated'' (or tensorized) versions of this measure, and write $x$, $x'$, $x''$, etc.\ for the canonical ``replicated'' random variables. Conditionally on $\bar x$, $W$ and $z$, these random variables are independent and each is distributed according to the Gibbs measure $\la \cdot \ra$. With a slight abuse of notation, we still denote this tensorized measure by $\la \cdot \ra$. That is, for every bounded measurable function $f : (\R^{N\times K})^2\to \R$, we denote
\begin{equation*}  %\label{e.}
\la f(x,x') \ra = \frac{1}{Z_N^2(t,h)} \int_{(\R^{N\times K})^2} f(x,x') e^{H_N(t,h,x) + H_N(t,h,x')} \, \d P_N(x) \, \d P_N(x'),
\end{equation*}
and so on for more than two replicas. 

\smallskip

As discussed in the introduction and shown in the appendix, the measure $\la \cdot \ra$ is the conditional measure of~$\bar x$ given $\mcl Y := (Y,Y')$. That is, for every bounded measurable function $f : \R^{N\times K} \to \R$, 
\begin{equation}  
\label{e.def.cond.measure}
\la f(x) \ra = \E \Ll[ f(\bar x) \ \big\vert \ \mcl Y\Rr] .
\end{equation}
In particular, we have
\begin{equation*}  %\label{e.}
\E \la f(x) \ra = \E \Ll[ f(\bar x) \Rr] .
\end{equation*}
Using the conditional independence between replicas, we also have, for any bounded measurable functions $f_1,f_2 : \R^{N\times K} \to \R$,
\begin{align*}  %\label{e.}
\E \la f_1(x) \, f_2 (x') \ra 
& = \E \Ll[ \la f_1(x) \ra \la f_2(x) \ra \Rr] 
\\
& = \E \Ll[ \la f_1(x) \ra f_2(\bar x) \Rr] .
\end{align*}
By the monotone class lemma, this implies that for any bounded measurable function $f : (\R^{N\times K})^2 \to \R$, 
\begin{equation}
\label{e.Nishi}
\E \la f(x,x') \ra = \E \la f(x,\bar x) \ra. 
\end{equation}
This identity can be generalized to more than two replicas: for instance,
\begin{equation*}  %\label{e.}
\E \la f(x,x',x'') \ra = \E \la f(x,x',\bar x) \ra.
\end{equation*}
This relation, regardless of the number of replicas involved, is often called the \emph{Nishimori identity}. This property will be the crucial ingredient allowing us to ``close the equation'' and show that the system stays in a replica-symmetric phase. We will repeatedly use it without further mention. Notice that we can also incorporate dependencies in $\mcl Y$, that is:
\begin{equation*}  %\label{e.}
\E \la f\Ll( \mcl Y , x,x' \Rr) \ra = \E \la f\Ll( \mcl Y,  x,\bar x\Rr) \ra,
\end{equation*}
and so on with more replicas.

\subsection{Matrix square root}
We record here some elementary properties of the matrix square root that will be useful in the sequel. Denote by $D_{\sqrt{h}}$ the differential of the square-root function, seen as a mapping from $S^K_{++}$ to $S^K$, at the point $h \in S^K_{++}$. This differential is a linear mapping on $S^K$ which satisfies, for every $h \in S^K_{++}$ and $a \in S_K$,
\begin{equation*}  %\label{e.}
D_{\sqrt{h}}(a) = \lim_{\ep \to 0} \ep^{-1} \Ll( \sqrt{h+\ep a} - \sqrt{h} \Rr) .
\end{equation*}
Notice that
\begin{multline*}  %\label{e.}
h + \ep a = \Ll( \sqrt{h + \ep a} \Rr) ^2 = \Ll(\sqrt{h} + \ep D_{\sqrt{h}}(a)   + o(\ep)\Rr)^2 
\\
= h + \ep \Ll( \sqrt{h} \, D_{\sqrt{h}}(a)  + D_{\sqrt{h}}(a)  \, \sqrt{h} \Rr) + o(\ep),
\end{multline*}
and thus
\begin{equation}  
\label{e.identif.Dsqrt}
\sqrt{h} \, D_{\sqrt{h}}(a)  + D_{\sqrt{h}}(a)  \, \sqrt{h} = a.
\end{equation}
Although we will not need this fact, it is interesting to note that this property characterizes $D_{\sqrt{h}}$. This tells us that any identity involving $D_{\sqrt{h}}$ is provable using only \eqref{e.identif.Dsqrt} as a definition of $D_{\sqrt{h}}$.
\begin{lemma}[unique Jordan inverse]
\label{l.unique.jordan}
Let $A \in S^K_{++}$ and $B,C \in \R^{K\times K}$ be such that 
\begin{equation}  
\label{e.unique.jordan}
AB + BA = AC + CA.
\end{equation}
Then $B = C$.
\end{lemma}
\begin{proof}
Without loss of generality, we can assume that $C = 0$ and that $A$ is a diagonal matrix, with positive eigenvalues $\lambda_1,\ldots, \lambda_K$. The condition \eqref{e.unique.jordan} then reads, for $B = (B_{kl})_{1 \le k,l \le K}$,
\begin{equation*}  %\label{e.}
\forall 1 \le k,l \le K, \qquad B_{kl}(\lambda_k + \lambda_l) = 0.
\end{equation*}
Since $\lambda_k + \lambda_l > 0$, this implies that $B = 0$.
\end{proof}
It follows from \eqref{e.identif.Dsqrt} that whenever $b \in S^K$, we have
\begin{equation}  
\label{e.sym.easy}
\Ll(D_{\sqrt{h}}(a)\sqrt{h}\Rr) \cdot b = \Ll(\sqrt{h}D_{\sqrt{h}}(a)\Rr) \cdot b = \frac{a\cdot b}{2}.
\end{equation}
When $b \in \R^{K\times K}$ is not assumed to be symmetric, we can  estimate the error in this relation in terms of the size of the antisymmetric part of $b$. Recall from \eqref{e.def.kappa} that we denote by $\kappa(h)$ the condition number of a matrix $h \in S^K_{+}$.
\begin{lemma}
\label{l.symmetrize}
There exists $C < \infty$ such that for every $h \in S^K_{++}$, $a \in S^K$ and $b \in \R^{K\times K}$, 
\begin{equation*}  %\label{e.}
\Ll| D_{\sqrt{h}}(a)\sqrt{h} \cdot b - \frac{a\cdot b}{2} \Rr| + \Ll| \sqrt{h} D_{\sqrt{h}}(a)\cdot b - \frac{a\cdot b}{2} \Rr| \le C \sqrt{\kappa(h)} \, |a|\, \Ll| b - b\t \Rr|. 
\end{equation*}
\end{lemma}
\begin{proof}
We first show that 
\begin{equation}
\label{e.norm.Dsqrt}
|D_{\sqrt{h}}(a)| \le C |a| \, \Ll|h^{-1}\Rr|^\frac 1 2.
\end{equation}
The proof is similar to that of Lemma~\ref{l.unique.jordan}. Without loss of generality, we assume that $h$ is diagonal, with positive eigenvalues $0 < \lambda_1\le \ldots\le\lambda_K$. Denoting $(\msf d_{kl})_{1 \le k,l \le K}$ the entries of the matrix $D_{\sqrt{h}}(a)$, and $a = (a_{kl})_{1 \le k,l\le K}$, the relation \eqref{e.identif.Dsqrt} reads
\begin{equation*}  %\label{e.}
\msf d_{kl} \Ll( \sqrt{\lambda_k} + \sqrt{\lambda_l} \Rr) = a_{kl}.
\end{equation*}
We deduce that
\begin{equation*}  %\label{e.}
\max_{1\le k,l \le K} |\msf d_{kl}| \le \lambda_1^{-\frac 1 2} \max_{1 \le k,l \le K} |a_{kl}| \le \Ll(\sum_{k = 1}^K \lambda_k^{-2}\Rr)^{\frac 1 4}\max_{1 \le k,l \le K} |a_{kl}|,
\end{equation*}
and thus that \eqref{e.norm.Dsqrt} holds, by equivalence of norms. The conclusion of the lemma then follows using the decomposition
\begin{equation*}  %\label{e.}
2 D_{\sqrt{h}}(a)\sqrt{h} \cdot b = D_{\sqrt{h}}(a)\sqrt{h} \cdot (b+b\t) + D_{\sqrt{h}}(a)\sqrt{h} \cdot (b-b\t),
\end{equation*}
and that by \eqref{e.sym.easy}, the first term on the right side is $a \cdot \frac{b + b\t}2 = a\cdot b$, since $a \in S^K$.
\end{proof}

\subsection{Proof of Proposition~\ref{p.approx.hj}}
For convenience, we now record some identities that follow from Gaussian integration by parts.

\begin{lemma}[Gaussian integration by parts]
\label{l.gaussian}
For every bounded measurable function $F : (\R^{N\times K})^2 \to \R^{N\times K}$, we have
\begin{equation}  
\label{e.gibp}
\E \la z \cdot F(x,\bar x) \ra = \E \la (x-x') \sqrt{h} \cdot F(x, \bar x) \ra 
\end{equation}
and
\begin{equation}
\label{e.gibp2}
\E \la \Ll(z \cdot F(x,\bar x)\Rr)^2 \ra = \E \la \Ll((x-x') \sqrt{h} \cdot F(x, \bar x)\Rr) \Ll( z \cdot F(x,\bar x) \Rr)  \ra + \E \la |F(x, \bar x)|^2 \ra ,
\end{equation}
while for $F : (\R^{N\times K})^3 \to \R^{N\times K}$,
\begin{equation}  
\label{e.gibp3}
\E \la z \cdot F(x,x',\bar x) \ra = \E \la (x+x'-2x'') \sqrt{h} \cdot F(x,x',\bar x) \ra .
\end{equation}
Finally,
\begin{equation}  
\label{e.gibp4}
\E \la x \cdot W  x \ra =\sqrt{\frac t N}\, \E \la |x\t x|^2 - |x\t \bar x|^2 \ra.
\end{equation}
\end{lemma}
\begin{proof}
We write $F = (F_{ik})_{1 \le i \le N, \, 1 \le k \le K}$ and notice that, by Gaussian integration by parts,
\begin{align*}  %\label{e.}
\E \Ll[ z_{ik} \la F_{ik}(x,\bar x) \ra \Rr] 
& = \E \Ll[ \dr_{z_{ik}} \la F_{ik}(x,\bar x) \ra \Rr] 
\\
& = \E \la F_{ik}(x,\bar x) \Ll( (x-x')\sqrt{h} \Rr)_{ik}  \ra .
\end{align*}
Summing over $(i,k)$, we obtain \eqref{e.gibp}. The proof of \eqref{e.gibp3} is similar: we write
\begin{align*}  %\label{e.}
\E \Ll[ z_{ik} \la F_{ik}(x,x',\bar x) \ra \Rr] 
& =  
\E \Ll[ \dr_{z_{ik}} \la F_{ik}(x,x',\bar x) \ra \Rr] 
\\
& = 
\E \la F_{ik}(x,x',\bar x) \Ll( (x + x' - 2 x'')\sqrt{h} \Rr)_{ik} \ra,
\end{align*}
and then sum over $(i,k)$.
For \eqref{e.gibp2}, we calculate first, for any bounded measurable $f :  (\R^{N\times K})^2 \to \R$ and any $(i,k), (j,l) \in \{1,\ldots,N\} \times \{1,\ldots, K\}$ satisfying $(i,k) \neq (j,l)$,
\begin{align*}  %\label{e.}
\E \Ll[ z_{ik} z_{jl} \la f(x,\bar x) \ra \Rr] 
& = \E \Ll[ z_{jl} \dr_{z_{ik}} \la f(x,\bar x) \ra \Rr] 
\\
& = \E \Ll[ z_{jl} \la f(x,\bar x) \Ll( (x-x')\sqrt{h} \Rr)_{ik}  \ra \Rr] .
\end{align*}
When $(i,k) = (j,l)$ this relation becomes
\begin{equation*}  %\label{e.}
\E \Ll[ z_{ik}^2 \la f(x,\bar x) \ra \Rr] =  \E \Ll[ z_{ik} \la f(x,\bar x) \Ll( (x-x')\sqrt{h} \Rr)_{ik}  \ra \Rr] + \E  \la f(x,\bar x) \ra .
\end{equation*}
Replacing $f(x,\bar x)$ by $F_{ik}(x,\bar x) F_{jl}(x,\bar x)$ and summing over all indices, we obtain \eqref{e.gibp2}.
Similarly, we observe that
\begin{equation*}  %\label{e.}
\E \la x \cdot W  x \ra = \E \la W \cdot x x\t\ra = \sqrt{\frac t N} \E \Ll[\la |xx\t|^2 \ra - |\la xx\t \ra|^2 \Rr],
\end{equation*}
and since
\begin{equation*}  %\label{e.}
\Ll| \la x x\t \ra \Rr| ^2 = \la x x\t \cdot \bar x \,\bar x\t \ra = \la |x\t \bar x |^2 \ra,
\end{equation*}
we obtain \eqref{e.gibp4}.
\end{proof}

We next present an intermediate result towards the proof of Proposition~\ref{p.approx.hj}, which is interesting on its own in that it displays the relevance of the question of assessing the concentration of the matrix $x\t \bar x \in \R^{K\times K}$.

\begin{lemma}
\label{l.first.approx.hj}
We have
\begin{equation}
\label{e.first.approx.hj}
\partial_t \bar F_N - 2 |\nabla \bar F_N|^2  = \frac 1 {2N^2} \E \la \Ll| x\t \bar x - \E\la x\t\bar x \ra \Rr|^2 \ra.
\end{equation}
Moreover, for every $t \ge 0$, the mapping $h \mapsto \bar F_N(t,h)$ is nondecreasing.
\end{lemma}
\begin{proof}
Starting with the derivative with respect to $t$, we have
\begin{equation}  
\label{e.formula.partialt}
\partial_t  F_N(t,{h}) = \frac 1 N\la \frac {1} {2\sqrt {t N}} \,  x \cdot W x +  \frac 1 N |x\t\bar x|^2 - \frac 1 {2N} |x\t x|^2 \ra.
\end{equation}
By \eqref{e.gibp4}, we deduce that 
\begin{equation}  
\label{e.identif.drt}
\partial_t \bar F_N(t,{h}) = \frac 1 {2N^2} \E \la |x\t \bar x|^2 \ra.
\end{equation}
We also have, for every $h \in S^K_{++}$ and $a \in S^K$,
\begin{equation}  
\label{e.grad.FN}
a \cdot \nabla F_N(t,h) = \frac 1 N \la  D_{\sqrt{h}}(a) \cdot x\t z + a \cdot x\t \bar x - \frac 1 2 a \cdot x\t x \ra.
\end{equation}
By \eqref{e.gibp}, we have 
\begin{align*}  %\label{e.}
\E \la   D_{\sqrt{h}}(a) \cdot x\t z \ra 
& =   \E \la D_{\sqrt{h}}(a)\cdot  \Ll(x\t (x-\bar x) \sqrt{h}\Rr) \ra.
\\
&  = 
  D_{\sqrt{h}}(a) \sqrt{h}\cdot \E \la x\t (x-\bar x)  \ra.
\end{align*}
Since $\E \la x\t (x-\bar x) \ra = \E \Ll[ \la x\t x \ra - \la x \ra\t \la x \ra \Rr] $ is a symmetric matrix, we can use \eqref{e.sym.easy} to infer that
\begin{equation*}  %\label{e.}
\E \la   D_{\sqrt{h}}(a) \cdot x\t z \ra  = \frac 1 2 \la a \cdot \Ll(x\t (x-\bar x) \Rr) \ra,
\end{equation*}
 and therefore
\begin{equation}
\label{e.a.grad.barFN}
a \cdot \nabla \bar F_N = \frac{1}{2N} \E \la a \cdot x\t \bar x \ra .
\end{equation}
Since 
\begin{equation*}  %\label{e.}
\E \la x\t \bar x \ra = \E \Ll[\la x \ra\t \la x \ra\Rr] \in S^K_+,
\end{equation*}
we obtain that 
\begin{equation}
\label{e.identif.grad}
\nabla \bar F_N = \frac 1{2N} \E \la x\t \bar x \ra,
\end{equation}
and, by Lemma~\ref{l.grad}, that the mapping $h \mapsto \bar F_N(t,h)$ is nondecreasing. 
%To see the latter point clearly, we may write, for every $a \in S^K_+$,
%\begin{equation*}  %\label{e.}
%a \cdot \nabla \bar F_N = \frac{1}{2N} \E  \Ll[ \la x \ra a \cdot \la x \ra \Rr]   = \frac{1}{2N} \E \Ll[ |\la x \ra \sqrt{a}|^2 \Rr] \ge 0.
%\end{equation*}
%
%We may write this in coordinates as
%\begin{equation*}  %\label{e.}
%\nabla \bar F_N(t,h)\cdot a = \frac{1}{2N} \E\la\sum_{i = 1}^N \sum_{{1 \le k,l \le K}} x_{ik} a_{kl}  \bar x_{il} \ra  ,
%\end{equation*}
%or directly as
Combining~\eqref{e.identif.drt} and \eqref{e.identif.grad} yields \eqref{e.first.approx.hj}.
%\jc{
%One may rightfully be surprised to see the nonsymmetric matrix $x\t \bar x$ appear in the expression above. We can obtain more symmetric formulas if we are willing to accept the appearance of square matrices of size $N$ rather than $K$. Indeed, we can also write the identity~\eqref{e.identif.drt} in the form
%\begin{equation*}  %\label{e.}
%\dr_t \bar F_N = \frac 1 {2N^2} \E \la \tr\Ll(\bar x\t xx\t \bar x \Rr)\ra = \frac 1 {2N^2}\E \Ll[ \Ll| \la xx\t \ra \Rr|^2 \Rr],
%\end{equation*}
%so that
%\begin{equation*}  %\label{e.}
%\partial_t \bar F_N - 2 |\nabla \bar F_N|^2  = \frac 1 {2N^2}\Ll(\E \Ll[ \Ll| \la xx\t \ra \Rr|^2 \Rr] - \Ll| \E \Ll[ \la x \ra\t \la x \ra \Rr]  \Rr|^2 \Rr).
%\end{equation*}
%}
\end{proof}

We are now ready to complete the proof of Proposition~\ref{p.approx.hj}.

\begin{proof}[Proof of Proposition~\ref{p.approx.hj}]
In view of Lemma~\ref{l.first.approx.hj}, we aim to show that 
\begin{multline}
\label{e.second.approx.hj}
\frac 1 {2N^2} \E \la \Ll| x\t \bar x - \E\la x\t\bar x \ra \Rr|^2 \ra 
\\
\le  C \kappa(h) N^{ - \frac 1 4 } \Ll( \Delta \bar F_N + C \Ll| h^{-1} \Rr|  \Rr) ^\frac 1 4 + C \E \Ll[ \Ll| \nabla F_N - \nabla \bar F_N \Rr| ^2 \Rr].
\end{multline}
We decompose the proof of this fact into five steps.

\smallskip

\emph{Step 1.} 
For every $a \in S^K$, we denote
\begin{equation*}  %\label{e.}
H_N'(a,h,x) :=  D_{\sqrt{h}}(a) \cdot x\t z + a \cdot x\t \bar x - \frac 1 2 a \cdot x\t x.
\end{equation*}
In this step, we show that, for every $a \in S^K$ and $h \in S^K_{++}$,
\begin{multline}
\label{e.control.fluctH}
 \E \la \big(  H_N'(a,h,x) - \E \la H_N'(a,h,x) \ra \big) ^2 \ra 
\le N a \cdot \nabla \Ll( a \cdot \nabla \bar F_N(t,h) \Rr)
\\
 + N^2 \, \E \Ll[ \Ll( a \cdot \nabla F_N(t,h) - a \cdot \nabla \bar F_N(t,h) \Rr) ^2 \Rr] + C N |a|^2 \, \Ll|h^{-1}\Rr| .
\end{multline}
In the expression on the right side, the quantity $\nabla \Ll( a \cdot \nabla \bar F_N(t,h) \Rr)$ is the gradient of the mapping $h \mapsto a\cdot \nabla \bar F_N(t,h)$, evaluated at $h$. In particular, $\nabla  \Ll( a \cdot \nabla \bar F_N(t,h) \Rr) \in S^K$. To show \eqref{e.control.fluctH}, we start from the variance decomposition
\begin{multline*}  %\label{e.}
 \E \la \big(  H_N'(a,h,x) - \E \la H_N'(a,h,x) \ra \big) ^2 \ra 
\\
= \E \la \Ll( H_N'(a,h,x) - \la H_N'(a,h,x) \ra\Rr)^2 \ra + \E \Ll[\Ll( \la H_N'(a,h,x) \ra - \E \la H_N'(a,h,x) \ra\Rr)^2 \Rr] 
.
\end{multline*}
By \eqref{e.grad.FN}, we have
\begin{equation*}  %\label{e.}
\E \Ll[\Ll( \la H_N'(a,h,x) \ra - \E \la H_N'(a,h,x) \ra\Rr)^2 \Rr] = N^2 \, \E \Ll[ \Ll( a\cdot \nabla F_N(t,h) - a\cdot \nabla \bar F_N(t,h) \Rr) ^2 \Rr].
\end{equation*}
For every $h \in S^K_{++}$ and $a,b \in S^K$, we write
\begin{equation*}  %\label{e.}
D^2_{\sqrt{h}}(a,b) := \lim_{\ep \to 0} \ep^{-1} \Ll( D_{\sqrt{h+\ep b}}(a) - D_{\sqrt{h}}(a) \Rr) ,
\end{equation*}
so that
\begin{multline}
\label{e.partialh2F}
a \cdot \nabla \Ll( a \cdot \nabla F_N(t,h) \Rr) = 
\\
 \frac{1}{N} \Ll(\la \Ll(H_N'(a,{h},x)\Rr)^2 \ra -\la H_N'(a,{h},x) \ra^2 \Rr) + \frac {1}{N} \la  D^2_{\sqrt{h}}(a,a) \cdot x\t z  \ra .
\end{multline}
Differentiating the identity \eqref{e.identif.Dsqrt}, we find that
\begin{equation}  
\label{e.relation.D2sqrt}
2\Ll( D_{\sqrt{h}}(a) \Rr) ^2 + \sqrt{h} D^2_{\sqrt{h}}(a,a) + D^2_{\sqrt{h}}(a,a) \sqrt{h} = 0.
\end{equation}
By Lemma~\ref{l.gaussian}, we also have that
\begin{equation*}  %\label{e.}
\E \la  D^2_{\sqrt{h}}(a,a) \cdot x\t z  \ra = D^2_{\sqrt{h}}(a,a) \sqrt{h} \cdot\E \la   x\t (x-x')  \ra.
\end{equation*}
Combining the two previous displays with the fact that the matrix $\E \la   x\t (x-x')  \ra$ is symmetric, we obtain that 
\begin{multline}
\label{e.partialh2barF}
a \cdot \nabla \Ll( a \cdot \nabla \bar F_N(t,h) \Rr) = 
\\
 \frac{1}{N} \E \la \Ll( H_N'(a,h,x) - \la H_N'(a,h,x) \ra\Rr)^2 \ra  - \frac {1}{N} \E\la  \Ll(D_{\sqrt{h}}(a)\Rr)^2 \cdot x\t (x-x')  \ra .
\end{multline}
By \eqref{e.norm.Dsqrt}, this completes the proof of \eqref{e.control.fluctH}.

\smallskip

\emph{Step 2.} We now aim to show that the variance of~$x\t \bar x$ is controlled by a finite sum over $a$ of variances of $H_N'(a,h,x)$ (or equivalently, by the supremum over $|a| \le 1$ of these variances, since $a \mapsto H_N'(a,h,x)$ is linear). In this step, we show that there exists a constant $C < \infty$ such that for every $a \in S^K$,
\begin{multline}
\label{e.relat.fluct}
  \E \la \Ll( a \cdot x\t \bar x - \E\la a \cdot x\t \bar x \ra \Rr)^2 \ra  
 \\
 \le  4 \E \la \big(  H_N'(a,h,x) - \E \la H_N'(a,h,x) \ra \big) ^2 \ra +  C |a|^2 \kappa(h)  \,\msf{Skew},
\end{multline}
where $\msf{Skew}$ is a quantity measuring the skewness of the matrix $x\t \bar x$:
\begin{equation}
\label{e.def.skew}
\msf{Skew} := N \Ll(\E \la \Ll| x\t \bar x - \bar x\t x \Rr|^2 \ra\Rr)^\frac 1 2.
\end{equation}
Since
\begin{equation*}  %\label{e.}
 \E\la a \cdot x\t \bar x \ra = 4 \E \la H_N'(a,h,x) \ra ,
\end{equation*}
it suffices to show that
\begin{equation}  
\label{e.relat.fluct.simpl}
 \E \la \Ll( a \cdot x\t \bar x \Rr)^2 \ra  \le  4 \E \la  H_N'(a,h,x) ^2 \ra + C  |a|^2 \kappa(h)  \,\msf{Skew}.
\end{equation}
By Lemmas~\ref{l.gaussian} and \ref{l.symmetrize} and \eqref{e.norm.Dsqrt}, we have
\begin{align*}  %\label{e.}
& \E \la \Ll(D_{\sqrt h}(a) \cdot x\t z\Rr)^2 \ra 
\\
& \qquad 
=
\E \la \Ll(D_{\sqrt h}(a) \sqrt{h} \cdot x\t (x-x')\Rr)\Ll(D_{\sqrt h}(a) \cdot x\t z\Rr) \ra 
+ 
\E \la \Ll|D_{\sqrt h}(a) x\Rr|^2 \ra
\\
& \qquad 
=
\E \la \Ll(D_{\sqrt h}(a) \sqrt{h} \cdot x\t (x-x')\Rr)\Ll(D_{\sqrt h}(a) \sqrt{h} \cdot x\t (x + x' - 2 \bar x)\Rr) \ra 
+ 
\E \la \Ll|D_{\sqrt h}(a) x\Rr|^2 \ra
\\
& \qquad 
\ge 
\frac 1 4\E \la \Ll(a\cdot x\t (x-x')\Rr)\Ll(a\cdot x\t (x + x' - 2 \bar x)\Rr) \ra 
%+ 
%\E \la \Ll|D_{\sqrt h}(a) x\Rr|^2 \ra 
- C  |a|^2 \kappa(h)  \,\msf{Skew}.
\end{align*}
The first term on the right side of the previous display can be rewritten as
\begin{equation*}  %\label{e.}
\frac 1 4\E \la \Ll(a\cdot x\t x\Rr)^2 \ra 
- \frac 1 2\E \la \Ll(a\cdot x\t x\Rr)\Ll(a\cdot x\t \bar x\Rr) \ra 
- \frac 1 4 \E \la (a \cdot x\t \bar x)^2 \ra
+ \frac 1 2 \E \la \Ll(a\cdot x\t \bar x\Rr)\Ll(a\cdot x\t x'\Rr) \ra .
\end{equation*}
Appealing again to Lemmas~\ref{l.gaussian} and \ref{l.symmetrize}, we can also write
\begin{align*}  %\label{e.}
&  \E \la 2\Ll(D_{\sqrt h}(a) \cdot x\t z\Rr) \Ll( a \cdot x\t \bar x - \frac 1 2 a \cdot x\t x \Rr) \ra
\\
& 
\qquad =  \E \la 2\Ll(D_{\sqrt h}(a)\sqrt{h} \cdot x\t (x-x')\Rr) \Ll( a \cdot x\t \bar x - \frac 1 2 a \cdot x\t x \Rr) \ra
\\
& 
\qquad \ge   \E \la \Ll(a \cdot x\t (x-x')\Rr) \Ll( a \cdot x\t \bar x - \frac 1 2 a \cdot x\t x \Rr) \ra - C  |a|^2 \sqrt{\kappa(h)}  \, \msf{Skew},
\end{align*}
and the first term on the right side is equal to
\begin{equation*}  %\label{e.}
-\frac 1 2 \E \la \Ll( a\cdot x\t x \Rr)^2 \ra  + \frac 3 2 \E \la \Ll(a\cdot x\t x\Rr)\Ll(a\cdot x\t \bar x\Rr) \ra  -  \E \la \Ll(a\cdot x\t \bar x\Rr)\Ll(a\cdot x\t x'\Rr) \ra .
\end{equation*}
Finally,
\begin{equation*}  %\label{e.}
\E \la \Ll( a \cdot x\t \bar x - \frac 1 2 a \cdot x\t x \Rr)^2 \ra
\\
= \frac 1 4 \E \la \Ll( a\cdot x\t x \Rr)^2 \ra - \E \la \Ll(a\cdot x\t x\Rr)\Ll(a\cdot x\t \bar x\Rr) \ra + \E \la \Ll( a \cdot x\t \bar x \Rr) ^2 \ra .
\end{equation*}
Summing the previous displays, we obtain that
\begin{equation}  
\label{e.identif.H2}
 \E \la  H_N'(a,h,x) ^2 \ra 
 \ge \frac 3 4 \E \la \Ll( a \cdot x\t \bar x \Rr) ^2 \ra - \frac 1 2  \E \la \Ll(a\cdot x\t \bar x\Rr)\Ll(a\cdot x\t x'\Rr) \ra - C |a|^2 \kappa(h) \,\msf{Skew}.
\end{equation}
By the Cauchy-Schwarz inequality,
\begin{equation*}  %\label{e.}
\Ll| \E \la \Ll(a\cdot x\t \bar x\Rr)\Ll(a\cdot x\t x'\Rr) \ra \Rr|  \le  \E \la \Ll( a \cdot x\t \bar x \Rr) ^2 \ra.
\end{equation*}
Combining the two previous displays yields \eqref{e.relat.fluct.simpl}, and therefore also~\eqref{e.relat.fluct}.

\smallskip

\emph{Step 3.} There remains to control the skew-symmetric part of $x\t \bar x$. Following the approach of \cite{bar19}, we decompose the argument into two steps. In this step, we find a convenient expression for the second derivative of $\bar F_N$, namely,
\begin{multline}  
\label{e.expr.drh2}
 a\cdot \nabla \Ll( a \cdot \nabla \bar F_N \Rr)
\\
 = \frac 1 {2N} \Ll( \E \la \Ll( a \cdot x\t \bar x \Rr) ^2\ra - 2 \E \la \Ll( a \cdot x\t \bar x \Rr) \Ll( a\cdot x\t x' \Rr) \ra + \E \Ll[ \la a\cdot x\t x'\ra^2 \Rr]  \Rr) .
\end{multline}
Differentiating \eqref{e.a.grad.barFN} gives
\begin{equation}  
\label{e.magic.is.coming}
a\cdot \nabla \Ll( a \cdot \nabla \bar F_N \Rr) = \frac 1 {2N} \E \la \Ll(a \cdot x\t \bar x \Rr) \Ll( H_N'(a,h,x) - H_N'(a,h,x') \Rr) \ra.
\end{equation}
Moreover, by Lemma~\ref{l.gaussian},
\begin{align*}  %\label{e.}
\E \la \Ll( a \cdot x\t \bar x \Rr) D_{\sqrt{h}}(a) \cdot x\t z \ra
= 
\E \la \Ll( a \cdot x\t \bar x \Rr) D_{\sqrt{h}}(a) \sqrt{h}\cdot x\t (x-x') \ra,
\end{align*}
while
\begin{equation*}  %\label{e.}
\E \la \Ll( a \cdot x\t \bar x \Rr) D_{\sqrt{h}}(a) \cdot (x')\t z \ra
= 
\E \la \Ll( a \cdot x\t \bar x \Rr) D_{\sqrt{h}}(a) \sqrt{h}\cdot (x')\t (x+x'-2x'') \ra.
\end{equation*}
Observe that
\begin{equation*}  %\label{e.}
x\t x - \Ll( x\t x' + (x')\t x \Rr) - (x')\t x' + (x')\t x'' + (x'')\t x'
\end{equation*}
is a symmetric matrix. Using the symmetry between the replicas $x'$ and $x''$, \eqref{e.sym.easy}, and then that $a$ is a symmetric matrix, we thus obtain that
\begin{align*}  %\label{e.}
& 
\E \la \Ll( a \cdot x\t \bar x \Rr) D_{\sqrt{h}}(a) \cdot \Ll(x-x'\Rr)\t z \ra
\\
& \qquad = 
\frac 1 2 \E \la \Ll( a \cdot x\t \bar x \Rr)  a\cdot \Ll( x\t x -  x\t x' - (x')\t x  - (x')\t x' + (x')\t x'' + (x'')\t x'\Rr)   \ra.
\\
& \qquad = \frac 1 2 \E \la  \Ll( a\cdot x\t \bar x \Rr)  a\cdot\Ll( x\t  x - (x')\t x'-2x\t x' + 2 (x')\t x''\Rr)   \ra .
\end{align*}
Combining this with \eqref{e.magic.is.coming} yields
\begin{align*}  %\label{e.}
 a\cdot \nabla \Ll( a \cdot \nabla \bar F_N \Rr)
& = 
\frac 1 {2N} \E \la  \Ll( a\cdot x\t \bar x \Rr)  a\cdot\Ll( -x\t x'  +  (x')\t x''+  x\t \bar x  - (x')\t \bar x \Rr)   \ra, 
%\\
%& = \frac 1 {2N} \Ll( \E \la \Ll( a \cdot x\t \bar x \Rr) ^2\ra - 2 \E \la \Ll( a \cdot x\t \bar x \Rr) \Ll( a\cdot x\t x' \Rr) \ra + \E \Ll[ \la a\cdot x\t x'\ra^2 \Rr]  \Rr) ,
\end{align*}
which is \eqref{e.expr.drh2}.

\smallskip

\emph{Step 4.} In this step, we show that
\begin{equation}
\label{e.skew.est}
\msf{Skew} \le C N^\frac 7 4  \Big(\sup_{\substack{v \in \R^{K\times 1} \\ |v| \le 1}}  (vv\t) \cdot \nabla \Ll( (vv\t)\cdot \nabla \bar F_N \Rr)  \Big)^\frac 1 4.
\end{equation}
For each pair of vectors $v,w \in \R^{K\times 1}$, we look for an upper bound on the quantity
\begin{equation}  
\label{e.needed.upper}
\E \la \Ll( vw\t \cdot \Ll(x\t \bar x - \la  x\t x' \ra\Rr) \Rr)  ^2 \ra = \E \la \Ll( vw\t \cdot \Ll(x\t x' - \la  x\t x' \ra\Rr) \Rr)  ^2 \ra.
\end{equation}
Since
\begin{equation*}  %\label{e.}
\mathrm{Span} \Ll(\Ll\{v w\t \ : \ v,w \in \R^{K\times 1} \Rr\} \Rr) = \R^{K\times K},
\end{equation*}
and since the matrix $\la x\t x'\ra = \la x\ra\t \la x\ra$ is symmetric, knowing that the quantity in~\eqref{e.needed.upper} is small would indeed tell us that the antisymmetric part of $x\t \bar x$ concentrates around~$0$. We will make use of the fact that for every $x,x' \in \R^{N\times K}$ and $v,w \in \R^{K\times 1}$,
\begin{equation}
\label{e.unpolarize}
\Ll( v w\t \cdot x\t x' \Rr) ^2 = \Ll( xvv\t x\t \Rr) \cdot \Ll( x'ww\t(x')\t \Rr) .
\end{equation}
This follows from
\begin{multline*}  %\label{e.}
\Ll( v w\t \cdot x\t x' \Rr) ^2 
 = \Ll( xv \cdot x'w \Rr) ^2
 = \Ll( v\t x\t  x'w \Rr) ^2
 \\
 = \Ll( v\t x\t  x'w \Rr)\cdot \Ll( v\t x\t  x'w \Rr)
 = \Ll( xvv\t x\t \Rr) \cdot \Ll( x'ww\t(x')\t \Rr) .
\end{multline*}
In particular,
\begin{align*}
\la \Ll( vw\t \cdot x\t x' \Rr) ^2 \ra
& = \la xvv\t x\t \ra \cdot \la xww\t x\t \ra,
\end{align*}
and
\begin{equation*}  %\label{e.}
\Ll( vw\t \cdot \la x\t x'\ra \Rr) ^2 = \Ll( \la x \ra vv\t \la x \ra\t \Rr) \cdot \Ll( \la x \ra w w\t \la x\ra^t \Rr) .
\end{equation*}
We can thus rewrite the quantity in \eqref{e.needed.upper} as
\begin{align*}  %\label{e.}
& \E \Ll[ \la \Ll( vw\t \cdot x\t x' \Rr) ^2 \ra - \Ll( vw\t \cdot \la x\t x'\ra \Rr) ^2 \Rr] 
\\
& \quad =  \E \Ll[\la xvv\t x\t \ra \cdot \la xww\t x\t \ra
- \Ll( \la x \ra vv\t \la x \ra\t \Rr) \cdot \Ll( \la x \ra w w\t \la x\ra^t \Rr) \Rr] 
\\
& \quad =  \E \Ll[ \Ll( \la xvv\t x\t\ra - \la x \ra vv\t \la x \ra\t\Rr)\cdot  \la xww\t x\t \ra 
+
 \la x \ra vv\t \la x \ra\t\cdot 
 \Ll(\la xww\t x\t \ra -  \la x \ra w w\t \la x\ra^t \Rr) \Rr] .
\end{align*}
By the triangle and the Cauchy-Schwarz inequalities, we thus see that
\begin{align*}  %\label{e.}
\msf{Skew} 
& =  N \Ll(\E \la \Ll| x\t \bar x - \bar x\t x \Rr|^2 \ra\Rr)^\frac 1 2
\\
& \le 2 N \Ll(\E \la \Ll| x\t \bar x - \la x\t x'\ra \Rr|^2 \ra\Rr)^\frac 1 2
\\
& \le C N \Big( \sup_{\substack{v,w \in \R^{K\times 1} \\ |v|,|w| \le 1}}\E \la \Ll( vw\t \cdot \Ll(x\t \bar x - \la  x\t x' \ra\Rr) \Rr)  ^2 \ra \Big)^\frac 1 2
\\
& \le C N^\frac 3 2\Big( \sup_{\substack{v \in \R^{K\times 1} \\ |v| \le 1}} \E\Ll[ \Ll| \la xvv\t x\t\ra - \la x \ra vv\t \la x \ra\t \Rr| ^2 \Rr]\Big)^\frac 1 4.
\end{align*}
(Recall that the constant $C$ is allowed to depend on $K$. Instead of the supremum over~$v$ and $w$, it may be more natural to write first a sum over $v,w \in \mcl V$ with $\mcl V$ a finite set such that $\{vw\t \ : \ v,w \in \mcl V\}$ spans $\R^{K\times K}$.)
Appealing again to \eqref{e.unpolarize}, we see that
\begin{equation*}  %\label{e.}
\E \Ll[\Ll| \la xvv\t x\t \ra \Rr|^2 \Rr]= \E \la \Ll( vv\t \cdot x\t x' \Rr) ^2 \ra,
\end{equation*}
\begin{equation*}  %\label{e.}
\E \Ll[\Ll| \la x \ra vv\t \la x \ra\t \Rr|^2 \Rr]= \E \Ll[\la vv\t \cdot x\t x' \ra ^2 \Rr].
\end{equation*}
A minor variant of \eqref{e.unpolarize} gives that 
\begin{equation*}  %\label{e.}
\Ll( vv\t \cdot x\t \bar x \Rr) \Ll( vv\t \cdot x\t x' \Rr) = \Ll( xvv\t x\t  \Rr) \cdot \Ll( \bar x v v\t (x')\t \Rr) ,
\end{equation*}
and thus
\begin{align*}  %\label{e.}
\E \Ll[ \la xvv\t x\t\ra \cdot \la x \ra vv\t \la x \ra\t  \Rr] 
& = \E \la \Ll( xvv\t x\t  \Rr) \cdot \Ll( \bar x v v\t (x')\t \Rr) \ra
\\
& = \E \la \Ll(vv\t\cdot x\t \bar x\Rr)\Ll(vv\t\cdot x\t x'\Rr) \ra.
\end{align*}
Combining these identities yields that
\begin{multline*}  %\label{e.}
\E\Ll[ \Ll| \la xvv\t x\t\ra - \la x \ra vv\t \la x \ra\t \Rr| ^2 \Rr] \\
=
\E \la \Ll( vv\t \cdot x\t x' \Rr) ^2 \ra - 2  \E \la \Ll(vv\t\cdot x\t \bar x\Rr)\Ll(vv\t\cdot x\t x'\Rr) \ra + \E \Ll[\la vv\t \cdot x\t x' \ra ^2 \Rr].
\end{multline*}
Using also \eqref{e.expr.drh2} completes the proof of \eqref{e.skew.est}.

\smallskip

\emph{Step 5.} We combine the results of the previous steps and complete the proof. Notice first that, since $\msf {Skew} \le C N^2$,
\begin{align}  
\notag
\E \la \Ll| x\t \bar x - \E\la x\t\bar x \ra \Rr|^2 \ra 
& \le C \sup_{\substack{a \in S^K \\ |a|\le 1}} \E \la \Ll( a \cdot x\t \bar x - \E\la a \cdot x\t \bar x \ra \Rr)^2 \ra + C N^{-2} \Ll( \msf{Skew} \Rr) ^2
\\
\label{e.split.sym.skew}
& \le C \sup_{\substack{a \in S^K \\ |a|\le 1}} \E \la \Ll( a \cdot x\t \bar x - \E\la a \cdot x\t \bar x \ra \Rr)^2 \ra +  C \msf{Skew}.
\end{align}
From \eqref{e.partialh2barF} and \eqref{e.norm.Dsqrt}, we see that, for every $a \in S^K$,
\begin{equation}  
\label{e.lowerbound.dr2barF}
a \cdot \nabla \Ll( a \cdot \nabla \bar F_N(t,h) \Rr) \ge -C \Ll|h^{-1}\Rr| \, |a|^2.
\end{equation}
Writing the Laplacian explicitly as a sum of derivatives in an orthogonal basis that contains a vector colinear to $a \in S^K$, we deduce that
\begin{equation*}  %\label{e.}
a \cdot \nabla \Ll( a \cdot \nabla \bar F_N(t,h) \Rr) \le |a|^2 \Ll( \Delta \bar F_N(t,h) + C \Ll|h^{-1}\Rr| \Rr) .
\end{equation*}
In particular, we can replace $a \cdot \nabla \Ll( a \cdot \nabla \bar F_N(t,h) \Rr)$ by $|a|^2 \Delta \bar F_N(t,h)$ on the right side of \eqref{e.control.fluctH}, up to a modification of the constant $C$, and rewrite \eqref{e.skew.est} as
\begin{equation*}  %\label{e.}
\msf{Skew} \le C N^\frac 7 4 \Ll( \Delta \bar F_N + C \Ll| h^{-1} \Rr|  \Rr)^\frac 1 4.
\end{equation*}
Combining these estimates with \eqref{e.relat.fluct} and \eqref{e.split.sym.skew}, we obtain that
\begin{multline*}  %\label{e.}
\E \la \Ll|x\t \bar x - \E \la x\t \bar x \ra \Rr|^2 \ra \le C \Ll( N\Delta_N \bar F_N + N^2 \E \Ll[ \Ll| \nabla F_N - \nabla \bar F_N \Rr| ^2 \Rr] + C N \Ll|h^{-1}\Rr|  \Rr) 
\\
+ C \kappa(h) N^\frac 7 4 \Ll( \Delta \bar F_N + C \Ll| h^{-1} \Rr|  \Rr) ^\frac 1 4.
\end{multline*}
Since the left side of this inequality is bounded, we can simplify this into \eqref{e.second.approx.hj}, thereby completing the proof.
\end{proof}
\begin{remark}  
\label{r.convex}
Contrary to the rank-one case, I do not know whether the mapping $h \mapsto \bar F_N(t,h)$ is convex. However, this mapping does satisfy a partial convexity property in the direction of positive semidefinite matrices. (We will not make use of this fact.) To make this point explicit, notice first that if $y,y',y''$ are three i.i.d.\ random vectors under $\la \cdot \ra$, then 
\begin{equation}
\label{e.positivity}
\la (y\cdot y')^2 \ra - 2 \la (y\cdot y')(y\cdot y'') \ra + \la y\cdot y' \ra^2 \ge 0.
\end{equation}
Indeed, this follows from the fact that the left side of \eqref{e.positivity} can be rewritten as
\begin{equation*}  %\label{e.}
\Ll| \la yy\t \ra - \la y \ra \la y \ra \t  \Rr| ^2.
\end{equation*}
Recalling \eqref{e.expr.drh2}, we see that, for every $a \in S^K_+$,
\begin{align*}  %\label{e.}
& a \cdot \nabla \Ll( a \cdot \nabla \bar F_N \Rr) 
\\
& \qquad = \frac 1 {2N} \E \Ll[\la \Ll( a \cdot x\t \bar x \Rr) ^2\ra - 2 \la \Ll( a \cdot x\t \bar x \Rr) \Ll( a\cdot x\t x' \Rr) \ra +  \la a\cdot x\t x'\ra^2 \Rr]  
\\
& \qquad = 
\frac 1 {2N} \E \Ll[\la \Ll( x\sqrt{a} \cdot \bar x\sqrt{a} \Rr) ^2\ra - 2 \la \Ll( x\sqrt{a} \cdot \bar x\sqrt{a} \Rr) \Ll( x\sqrt{a}\cdot x'\sqrt{a} \Rr) \ra +  \la x\sqrt{a}\cdot x'\sqrt{a}\ra^2 \Rr] .
\end{align*}
Viewing $x\sqrt{a}$ as a vector with $NK$ entries and applying \eqref{e.positivity}, we obtain that this quantity is nonnegative. However, since this reasoning requires that we take the square root of $a$, it only applies to the situation when $a \in S^K_+$. (By symmetry, the case when $-a \in S^K_+$ is of course also covered.)
\end{remark}

\subsection{Derivative and concentration estimates}

We next record simple derivative and concentration estimates. We denote
\begin{equation*}  %\label{e.}
\|W\|_{\ell^2 \to \ell^2} := \sup \{|W x|  \ : \ x \in \R^{N\times 1}, \ |x|\le 1\}.
\end{equation*}

\begin{lemma}[Derivative estimates]
\label{l.gradient}
There exists a constant $C < \infty$ such that the following estimates hold uniformly over $\R_+ \times S^K_+$:
\begin{equation}  
\label{e.gradest.barF}
|\partial_t \bar F_N| + |\nabla \bar F_N|  
%+ N^{-1} |\partial_h^2 \bar F_N| + N^{-2} |\partial_h^3 \bar F_N| 
\le C,
\end{equation}
\begin{equation}  
\label{e.gradest.F}
|\partial_t F_N| \le C + \frac{C\|W\|_{\ell^2 \to \ell^2}}{\sqrt{Nt}}, \quad \text{ and } \quad |\nabla F_N| \le C + \frac{C|z| \, \Ll| h^{-1} \Rr|^\frac 1 2}{\sqrt{N}}.
\end{equation}
Moreover,  for every $a \in S^K$,
\begin{equation}  
\label{e.conv.F}
a \cdot \nabla \Ll( a \cdot \nabla \bar F_N \Rr)  \ge - C|a|^2 \, \Ll| h^{-1} \Rr|, \quad \text{ and } \quad a \cdot \nabla \Ll( a \cdot \nabla F_N \Rr)  \ge - \frac{C|a|^2 \, |z| \, \Ll| h^{-1} \Rr|^\frac 3 2}{\sqrt{N}}.
\end{equation}
\end{lemma}
\begin{proof}
The estimates in  \eqref{e.gradest.barF} and \eqref{e.gradest.F} follow from \eqref{e.identif.drt}, \eqref{e.identif.grad}, \eqref{e.formula.partialt}, \eqref{e.grad.FN} and \eqref{e.norm.Dsqrt}. (Recall that the constants are allowed to depend on $K$.)
The first part of \eqref{e.conv.F} is a consequence of \eqref{e.partialh2barF} and \eqref{e.norm.Dsqrt}. To obtain the second part of \eqref{e.conv.F}, we see from \eqref{e.partialh2F} that it suffices to establish that 
\begin{equation}  
\label{e.norm.D2sqrt}
\Ll|D^2_{\sqrt{h}}(a,a) \Rr| \le C |a|^2 \, \Ll| h^{-1} \Rr| ^\frac 3 2.
\end{equation}
Up to a change of basis, we may assume that the matrix $h$ is diagonal, with eigenvalues $0 < \lambda_1 \le \cdots \le \lambda_K$. Denoting by $(\msf d_{kl})_{1 \le k,l \le K}$ and $(\msf d_{kl}')_{1 \le k,l \le K}$ the entries of the matrices $D_{\sqrt{h}}(a)$ and $D^2_{\sqrt{h}}(a)$ respectively, we see from \eqref{e.relation.D2sqrt} that for every $k,l \in \{1,\ldots,K\}$,
\begin{equation*}  %\label{e.}
\Ll(\sqrt{\lambda_k} + \sqrt{\lambda_l} \Rr)\msf d'_{kl} = -2 \sum_{m = 1}^K \msf d_{km} \msf d_{ml} .
\end{equation*}
We thus obtain \eqref{e.norm.D2sqrt} using \eqref{e.norm.Dsqrt}.
\end{proof}

We now turn to a concentration estimate. We simply state an $L^2$ estimate with a suboptimal exponent, since this is sufficient for our purposes, but point out that it is classical to improve upon this. 
\begin{lemma}[Concentration of free energy]
There exists $\al > 0$ and, for every compact set $V \subset \R_+\times S^K_+$, a constant $C < \infty$ such that for every $N \in \N$,
\label{l.concentration}
\begin{equation*}  %\label{e.}
\E \Ll[ \Ll\| F_N - \bar F_N \Rr\|_{L^\infty(V)}^2  \Rr] \le C N^{-\alpha}.
\end{equation*}
\end{lemma}
\begin{proof}
The proof is essentially the same as that of \cite[Lemma~3.2]{HJinfer}, so we only briefly sketch the argument. First, using the Efron-Stein and the Gaussian Poincar\'e inequalities, we verify that for every $M \ge 1$, there exists $C < \infty$ such that for every $t \le M$ and $|h| \le M$,
\begin{equation*}  %\label{e.}
\E \Ll[ \Ll( F_N - \bar F_N \Rr)^2(t,h) \Rr] \le C N^{-1}. 
\end{equation*}
We next use \eqref{e.gradest.barF} and \eqref{e.gradest.F} to assert that $F_N - \bar F_N$ is $\frac 1 2$-H\"older continuous, with a random H\"older seminorm that has finite moments of every order. We then write, for every $\ep \in (0,1]$,
\begin{equation*}  %\label{e.}
 \E \Ll[ \sup_{t \le M, |h| \le M} \Ll( F_N - \bar F_N \Rr)^2(t,h)  \Rr] 
\le C \sqrt{\ep} + \E \Ll[ \sup_{(t,h) \in A_\ep} \Ll( F_N - \bar F_N \Rr)^2(t,h)  \Rr],
\end{equation*}
where $A_\ep$ is an $\ep$-net of the set $\{(t,h) \in \R_+ \times S^K_+ \ : \ t \le M \text{ and } |h| \le M\}$. We can choose $A_\ep$ in such a way that $|A_\ep| \le C \ep^{-1-\KK}$, and thus, by a union bound,
\begin{equation*}  %\label{e.}
 \E \Ll[ \sup_{t \le M, |h| \le M} \Ll( F_N - \bar F_N \Rr)^2(t,h)  \Rr] \le C \sqrt{\ep} + C \ep^{-1-\KK} \, N^{-1}. 
\end{equation*}
Optimizing over $\ep$ leads to the desired result. 
\end{proof}

%
%
%
%%%%%%%%%%%%%%%%%%%%%%%%%%%%
%%%%%%%%%%%%%%%%%%%%%%%%%%%%
%
%
%

\section{Convergence to weak solution}
\label{s.convergence}

We now show how Proposition~\ref{p.approx.hj}, together with the concentration estimate in Lemma~\ref{l.concentration}, implies Theorem~\ref{t.hj}. The argument is an adaptation of the proof of uniqueness of weak solutions of \eqref{e.hj}, see Proposition~\ref{p.uniqueness}. One minor simplification comes from the fact that we can assert the local semiconvexity property uniformly in a neighborhood of the region $t = 0$. Indeed, the local semiconvexity of $\bar F_N$ given by \eqref{e.conv.F} does not degenerate as $t \to 0$, and the corresponding property for the limit solution $f$ is provided by \eqref{e.local.semiconv.fth}. 

\begin{proof}[Proof of Theorem~\ref{t.hj}]
We decompose the proof into four steps. 

\smallskip

\emph{Step 1.} We set up the argument, find an approximate equation for the difference between $\bar F_N$ and the candidate limit, and state elementary bounds, paralleling Steps~1 and 2 of the proof of Proposition~\ref{p.uniqueness}. Denote by $f$ the weak solution of \eqref{e.hj} with initial condition $f(0,\cdot) = \psi = \bar F_1(0,\cdot)$. We set 
\begin{equation*}  %\label{e.}
w_N := \bar F_N - f, \qquad \text{and} \qquad \Err := \dr_t \bar F_N - 2 |\nabla \bar F_N|^2. 
\end{equation*}
The following holds almost everywhere in $\R_+ \times S^K_+$:
\begin{align*}  %\label{e.}
\dr_t w_N & = 2 |\nabla \bar F_N|^2 - 2 |\nabla f|^2 + \Err
\\
& = 2 \Ll( \nabla \bar F_N + \nabla f \Rr) \cdot \nabla w_N + \Err.
\end{align*}
We denote 
\begin{equation*}  %\label{e.}
\b_N := 2 \Ll( \nabla \bar F_N + \nabla f \Rr).
\end{equation*}
Let $\phi \in C^\infty(\R)$ be a nonnegative smooth function satisfying $|\phi'| \le 1$ and $\phi(0) = 0$, and set $v_N := \phi(w_N)$. We have
\begin{equation*}  %\label{e.}
\dr_t v_N - \b_N \cdot \nabla v_N = \phi'(w_N) \Err \qquad \text{a.e. in } \R_+ \times S^K_+.
\end{equation*}
We define $\zeta_\ep$ as in \eqref{e.def.zetaeps}, $f_\ep$ as in \eqref{e.def.f.eps}, and 
\begin{equation*}  %\label{e.}
\b_{N,\ep} := 2 \Ll( \nabla \bar F_N + \nabla f_\ep \Rr) ,
\end{equation*}
so that
\begin{equation*}  %\label{e.}
\dr_t v_N - \b_{N,\ep} \cdot \nabla v_N = (\b_N - \b_{N,\ep})\cdot \nabla v_N + \phi'(w_N) \Err \qquad \text{a.e. in } \R_+ \times S^K_+.
\end{equation*}
By \eqref{e.conv.F} with $N = 1$ and \eqref{e.local.semiconv.fth}, there exists $C < \infty$ such that for every $\de \in (0,1]$ and $\ep \in [0,\frac \de 2]$, 
\begin{equation*}
%\label{e.delta.fep}
\Delta f_\ep + C \de^{-1} \ge 0 \qquad \text{on } \R_+ \times S^K_{+\de}.
\end{equation*}
By \eqref{e.conv.F}, there exists $C < \infty$ such that for every $N \in \N$ and $\de \in (0,1]$, 
\begin{equation*}  %\label{e.}
\Delta \bar F_N + C \de^{-1} \ge 0 \qquad \text{on } \R_+ \times S^K_{+\de}.
\end{equation*}
We thus obtain that for every $N \in \N$, $\de \in (0,1]$ and $\ep \in [0,\frac \de 2]$, 
\begin{equation}
\label{e.divbN}
\nabla \cdot  \b_{N,\ep} + C \de^{-1} \ge 0 \qquad \text{on } \R_+ \times S^K_{+\de}.
\end{equation}

\smallskip

\emph{Step 2.} We essentially reproduce the arguments in Step 3 of the proof of Proposition~\ref{p.uniqueness}, temporarily leaving aside the new error term $\Err$. We denote
\begin{equation*}  %\label{e.}
R := 1 + 2 \|\nabla f\|_{L^\infty(\R_+\times S^K_+)} + 2 \sup_{N \in \N} \Ll( \|\nabla \bar F_N\|_{L^\infty(\R_+\times S^K_+)} \Rr),
\end{equation*}
which is finite by \eqref{e.gradest.barF}. We fix $T \ge 1$ and define, for every $t \in [0, \frac T 2]$, the sets $B_\de(t)$, $\dr_+ B_\de(t)$, and $\dr_0 B_\de(t)$  displayed in \eqref{e.def.Bde}-\eqref{e.def.bBde}, as well as
\begin{equation*}  %\label{e.}
J_{\de,N}(t) := \int_{B_\de(t)} v_N(t,\cdot). 
\end{equation*}
The function $J_{\de,N}$ is Lipschitz, and for almost every $t \in [0,\frac T 2]$,
\begin{align*}  %\label{e.}
& \dr_t J_{\de,N}(t) 
\\
& \quad = \int_{B_\de(t)} \dr_t v_N(t,\cdot) - R \int_{\dr_+ B_\de(t)} v_N(t,\cdot)
\\
& \quad = 
\int_{B_\de(t)} \Ll(\b_{N,\ep} \cdot \nabla v_N + (\b_N - \b_{N,\ep})\cdot \nabla v_N + \phi'(w_N) \Err\Rr)(t,\cdot) - R \int_{\dr_+ B_\de(t)} v_N(t,\cdot).
\end{align*}
By \eqref{e.ae.lim} and the dominated convergence theorem, we have
\begin{equation*}  %\label{e.}
\lim_{\ep \to 0} \int_{B_\de(t)}  \Ll((\b_N - \b_{N,\ep})\cdot \nabla v_N\Rr)(t,\cdot) = 0. 
\end{equation*}
Integrating by parts, we see that
\begin{equation*}  %\label{e.}
\int_{B_\de(t)} \Ll(\b_{N,\ep} \cdot \nabla v_N\Rr)(t,\cdot) = -\int_{B_\de(t)} \Ll(v_N \nabla \cdot \b_{N,\ep}\Rr)(t,\cdot) + \int_{\dr{B_\de(t)}} \Ll(v_N \b_N\cdot \n\Rr)(t,\cdot),
\end{equation*}
where $\n$ is the unit outer normal to $B_\de(t)$. By \eqref{e.divbN}, we have
\begin{equation*}  %\label{e.}
-\int_{B_\de(t)} \Ll(v_N \nabla \cdot \b_{N,\ep}\Rr)(t,\cdot) \le C \de^{-1} J_{\de,N}(t). 
\end{equation*}
We decompose the boundary integral into
\begin{equation*}  %\label{e.}
\int_{\dr{B_\de(t)}} \Ll(v_N \b_N\cdot \n\Rr)(t,\cdot) = \int_{\dr_+{B_\de(t)}} \Ll(v_N \b_N\cdot \n\Rr)(t,\cdot) + \int_{\dr_0{B_\de(t)}} \Ll(v_N \b_N\cdot \n\Rr)(t,\cdot).
\end{equation*}
By the definition of $R$, we have 
\begin{equation*}  %\label{e.}
 \int_{\dr_+{B_\de(t)}} \Ll(v_N \b_N\cdot \n\Rr)(t,\cdot) \le R  \int_{\dr_+{B_\de(t)}} v_N (t,\cdot).
\end{equation*}
Since both $\bar F_N$ and $f$ are nondecreasing in $h$, we infer from \eqref{e.equiv.pos} that
\begin{equation*}  %\label{e.}
\int_{\dr_0{B_\de(t)}} \Ll(v_N \b_N\cdot \n\Rr)(t,\cdot) \le 0.
\end{equation*}

\smallskip

\emph{Step 3.} 
There remains to estimate the contribution of the error term $\Err$. By Proposition~\ref{p.approx.hj}, and since $|\phi'| \le 1$, we have 
\begin{multline}  
\label{e.two.int}
\int_{B_\de(t)} \Ll( \phi'(w_N) \Err\Rr)(t,\cdot) 
\le C \de^{-1} N^{-\frac 1 4}\int_{B_\de(t)}   \Ll( \Delta \bar F_N + C \de^{-1}  \Rr) ^\frac 1 4(t,\cdot) 
\\
+C \int_{B_\de(t)} \E \Ll[ \Ll| \nabla F_N - \nabla \bar F_N \Rr| ^2 \Rr](t,\cdot) ,
\end{multline}
where we allow the multiplicative constant to depend also on $R$ and $T$. We estimate each of these two integrals in turn. By Jensen's inequality,
\begin{equation*}  %\label{e.}
\int_{B_\de(t)}   \Ll( \Delta \bar F_N + C \de^{-1}  \Rr) ^\frac 1 4(t,\cdot) \le \Ll(C \de^{-1} + \int_{B_\de(t)}   \Delta \bar F_N (t,\cdot) \Rr)^\frac 1 4,
\end{equation*}
and moreover, by integration by parts and \eqref{e.gradest.barF},
\begin{equation*}  %\label{e.}
\int_{B_\de(t)}   \Delta \bar F_N (t,\cdot) \le C.
\end{equation*}
Turning to the second integral on the right side of \eqref{e.two.int}, we introduce the notation $V := \{(t,h) \ : \ t \le T, |h| \le RT\}$, and integrate by parts and use \eqref{e.gradest.barF} again to get
\begin{align*}  %\label{e.}
& \int_{B_\de(t)} |\nabla F_N - \nabla \bar F_N|^2(t,\cdot) 
\\
& \quad = \int_{\dr B_\de(t)} \Ll((F_N - \bar F_N)\nabla(F_N - \bar F_N)\cdot \n\Rr)(t,\cdot) 
 -\int_{B_\de(t)} \Ll((F_N - \bar F_N)\Delta(F_N - \bar F_N)\Rr)(t,\cdot) 
\\
& \quad \le \|F_N - \bar F_N\|_{L^\infty(V)} \Ll( C + \int_{B_\de(t)} \Ll| \Delta (F_N - \bar F_N) \Rr| (t,\cdot)\Rr).
\end{align*}
We can then write
\begin{multline*}  %\label{e.}
\int_{B_\de(t)} \Ll| \Delta (F_N - \bar F_N) \Rr| (t,\cdot) 
\\
\le C \de^{-\frac 3 2} \Ll( 1 + \frac{|z|}{\sqrt{N}}  \Rr) + \int_{B_\de(t)} \Ll| \Delta (F_N - \bar F_N) + C \de^{-\frac 3 2} \Ll( 1 + \frac{|z|}{\sqrt{N}}  \Rr)\Rr| (t,\cdot) .
\end{multline*}
We next observe that for $C < \infty$ sufficiently large, the quantity between absolute values above is nonnegative, by \eqref{e.conv.F}. Integrating by parts and using \eqref{e.gradest.F}, we obtain that
\begin{equation*}  %\label{e.}
\int_{B_\de(t)} \Ll| \Delta (F_N - \bar F_N) \Rr| (t,\cdot)  \le C \de^{-\frac 3 2} \Ll( 1 + \frac{|z|}{\sqrt{N}}  \Rr).
\end{equation*}
Summarizing, and using the Cauchy-Schwarz inequality, we conclude that
\begin{equation*}  %\label{e.}
\int_{B_\de(t)} \E \Ll[ \Ll| \nabla F_N - \nabla \bar F_N \Rr| ^2 \Rr](t,\cdot) \le C \de^{-\frac 3 2} \E  \Ll[  \|F_N - \bar F_N\|_{L^\infty(V)}^2 \Rr] ^\frac 1 2 ,
\end{equation*}
and thus, by Lemma~\ref{l.concentration}, that there exists an exponent $\al \in (0,\frac 1 4]$ such that
\begin{equation*}  %\label{e.}
\int_{B_\de(t)} \E \Ll[ \Ll| \nabla F_N - \nabla \bar F_N \Rr| ^2 \Rr](t,\cdot) \le C \de^{-\frac 3 2} N^{-\al}.
\end{equation*}

\smallskip

\emph{Step 4.} We conclude the proof. Combining the results of the two previous steps, 
we obtain that almost everywhere in $[0,\frac T 2]$, we have
\begin{equation*}  %\label{e.}
\dr_t J_{N,\de} \le C \de^{-1} J_{\de,N} + C \de^{-\frac 3 2} N^{-\al}, 
\end{equation*}
that is,
\begin{equation*}  %\label{e.}
\dr_t \Ll(  \exp \Ll(-C\de^{-1} t\Rr) J_{\de,N} \Rr) \le C \exp \Ll(-C\de^{-1} t\Rr) \de^{-\frac 3 2} N^{-\al}.
\end{equation*}
Since $\phi(0) = 0$ and $w_N(0,\cdot) = 0$, this implies that for every $t \in [0,\frac T 2]$,
\begin{equation*}  %\label{e.}
\int_{B_\de(t)} \phi(\bar F_N - f)(t,\cdot) \le C  \exp(C\de^{-1}) \de^{-\frac 3 2}N^{-\alpha}.
\end{equation*}
(Recall that we allow the constant $C$ to depend on $T$.) We may as well absorb the term $\de^{-\frac 3 2}$ into the exponential. Since this estimate is valid uniformly over nonnegative $\phi \in C^\infty(\R)$ satisfying $\phi(0) = 0$ and $|\phi'| \le 1$, we deduce that for every $t \in [0,\frac T 2]$,
\begin{equation*}  %\label{e.}
\int_{B_\de(t)} \Ll| \bar F_N - f \Rr| (t,\cdot) \le C \exp ( C \de^{-1}) N^{-\al} .
\end{equation*}
Since the functions $\bar F_N$ and $f$ are locally bounded, uniformly over $N$, and the measure of the set $B_0(t) \setminus B_\de(t)$ is bounded by $C\de$, this implies that for every $t \in [0,\frac T 2]$,
\begin{equation*}  %\label{e.}
\int_{B_0(t)}\Ll| \bar F_N - f \Rr| (t,\cdot) \le C \de + C \exp ( C \de^{-1}) N^{-\al} .
\end{equation*}
We select $\de := C \log^{-1} N$, for a sufficiently large constant $C$, so that for every $t \in [0,\frac T 2]$,
\begin{equation*}  %\label{e.}
\int_{B_0(t)}\Ll| \bar F_N - f \Rr| (t,\cdot) \le \frac{C}{\log N}. 
\end{equation*}
This completes the proof of Theorem~\ref{t.hj}.
\end{proof}

%
%
%
%%%%%%%%%%%%%%%%%%%%%%%%%%%%
%%%%%%%%%%%%%%%%%%%%%%%%%%%%
%
%
%

\appendix

\section{Computation of the conditional law}
\label{s.nishi}

We denote
\begin{equation}  
\label{e.newdef.Y}
\mcl Y = (Y,Y') = \Ll( \sqrt{\frac{t}{N}} \, \bar x \, \bar x\t + W, \  \bar x\sqrt{h} + z\Rr).
\end{equation}
In this appendix, we verify that the conditional law of $\bar x$ given $\mcl Y$ is given by
\begin{equation}
\label{e.def.cond.law}
\frac{e^{H_N(t,h,x)} \, \d P_N(x)}{\int_{\R^{N\times K}} e^{H_N(t,h,x')} \, \d P_N(x')}.
\end{equation}
(See also~\eqref{e.def.cond.measure} for an equivalent statement.) 
For every bounded measurable functions $f$ and $g$, we can write~$\E \Ll[ f(\bar x) g(\mcl Y) \Rr]$, up to a normalization constant that depends neither on $f$ nor on $g$, as
\begin{equation*}
 \int f(x) g \Ll(  \sqrt{\frac{t}{N}} \, x \, x\t + W, \   x\sqrt{h} + z \Rr) \, \exp \Ll( - \frac {|W|^2} 2  - \frac{|z|^2}{2}  \Rr) \, \d W \, \d z \, \d P_N(x),
\end{equation*}
with the shorthand notation $\d W := \prod_{i,j} \d W_{ij}$ and $\d z := \prod_{i,k} \d z_{ik}$. A change of variables allows to rewrite the expression above as
\begin{equation*} 
\int f(x) g(\mcl Y) 
 \exp \Ll( -\frac 1 2 \Ll| Y - \sqrt{\frac t N} x\, x\t \Rr|^2 - \frac 1 2 \Ll|Y' -  x\sqrt{h}\Rr|^2 \Rr) \, \d \mcl Y \, \d P_N(x). 
\end{equation*}
Denoting the exponential factor above by $\mcl E(x,\mcl Y)$, we thus obtain that the law of $\mcl Y$ is the law with density given, up to a normalization constant, by
\begin{equation*}  %\label{e.}
\bar {\mcl E}(\mcl Y) := \int \mcl E(x,\mcl Y) \, \d P_N(x),
\end{equation*}
and that, denoting by $c$ the normalization constant,
\begin{equation}  
\label{e.joint.law}
\E \Ll[ f(\bar x) g(\mcl Y) \Rr] = c \int f(x) \frac{\mcl E(x,\mcl Y)}{\bar {\mcl E}(\mcl Y)} \, \d P_N(x)  \, g(\mcl Y) \, \bar {\mcl E}(\mcl Y) \, \d \mcl Y.
\end{equation}
The conditional law of $\bar x$ given $\mcl Y$ is thus the probability measure given by
\begin{equation*}  %\label{e.}
\frac{\mcl E(x,\mcl Y)}{\bar{\mcl E}(\mcl Y)} \, \d P_N(x),
\end{equation*}
and this quantity can indeed be rewritten in the form of \eqref{e.newdef.Y}.

\smallskip

\noindent \textbf{Acknowledgements} I would like to warmly thank Jean Barbier for stimulating discussions and for telling me about the results of \cite{bar19} prior to their publication. I was partially supported by the ANR grants LSD (ANR-15-CE40-0020-03) and Malin (ANR-16-CE93-0003) and by a grant from the NYU--PSL Global Alliance.

\small
\bibliographystyle{abbrv}
\bibliography{HJinfer}

\newcommand{\noop}[1]{} \def\cprime{$'$}
\begin{thebibliography}{10}

\bibitem{bar19}
J.~Barbier.
\newblock Overlap matrix concentration in optimal {B}ayesian inference,
  \noop{2019}{preprint, arXiv:1904.02808}.

\bibitem{bar1}
J.~Barbier, M.~Dia, N.~Macris, F.~Krzakala, T.~Lesieur, and L.~Zdeborov\'{a}.
\newblock Mutual information for symmetric rank-one matrix estimation: a proof
  of the replica formula.
\newblock In {\em Advances in Neural Information Processing Systems 29}, pages
  424--432, 2016.

\bibitem{bm1}
J.~Barbier and N.~Macris.
\newblock The adaptive interpolation method: a simple scheme to prove replica
  formulas in {B}ayesian inference.
\newblock {\em Probab. Theory Related Fields}, in press.

\bibitem{bmm}
J.~Barbier, N.~Macris, and L.~Miolane.
\newblock The layered structure of tensor estimation and its mutual
  information.
\newblock In {\em 55th Annual Allerton Conference on Communication, Control,
  and Computing}, pages 1056--1063. IEEE, 2017.

\bibitem{benton}
S.~H. Benton, Jr.
\newblock {\em The {H}amilton-{J}acobi equation}.
\newblock Academic Press, New York-London, 1977.
\newblock A global approach, Mathematics in Science and Engineering, Vol. 131.

\bibitem{ch06}
P.~Carmona and Y.~Hu.
\newblock Universality in {S}herrington-{K}irkpatrick's spin glass model.
\newblock {\em Ann. Inst. H. Poincar\'{e} Probab. Statist.}, 42(2):215--222,
  2006.

\bibitem{dm}
Y.~Deshpande and A.~Montanari.
\newblock Information-theoretically optimal sparse {PCA}.
\newblock In {\em IEEE International Symposium on Information Theory}, pages
  2197--2201, 2014.

\bibitem{dou65}
A.~Douglis.
\newblock Solutions in the large for multi-dimensional, non-linear partial
  differential equations of first order.
\newblock {\em Ann. Inst. Fourier (Grenoble)}, 15(fasc. 2):1--35, 1965.

\bibitem{eak}
A.~{El Alaoui} and F.~Krzakala.
\newblock Estimation in the spiked {W}igner model: a short proof of the replica
  formula, \noop{2018}{preprint, arXiv:1801.01593}.

\bibitem{evans}
L.~C. Evans.
\newblock {\em Partial differential equations}, volume~19 of {\em Graduate
  Studies in Mathematics}.
\newblock American Mathematical Society, Providence, RI, second edition, 2010.

\bibitem{kru66}
S.~N. Kru\v{z}kov.
\newblock Generalized solutions of nonlinear equations of the first order with
  several variables. {I}.
\newblock {\em Mat. Sb. (N.S.)}, 70 (112):394--415, 1966.

\bibitem{kru67}
S.~N. Kru\v{z}kov.
\newblock Generalized solutions of nonlinear equations of the first order with
  several independent variables. {II}.
\newblock {\em Mat. Sb. (N.S.)}, 72 (114):108--134, 1967.

\bibitem{lm}
M.~Lelarge and L.~Miolane.
\newblock Fundamental limits of symmetric low-rank matrix estimation.
\newblock {\em Probab. Theory Related Fields}, 173(3-4):859--929, 2019.

\bibitem{lkz}
T.~Lesieur, F.~Krzakala, and L.~Zdeborov{\'a}.
\newblock Phase transitions in sparse {PCA}.
\newblock In {\em IEEE International Symposium on Information Theory}, pages
  1635--1639, 2015.

\bibitem{HJinfer}
J.-C. Mourrat.
\newblock Hamilton-{J}acobi equations for mean-field disordered systems,
  \noop{2018}{preprint, arXiv:1811.01432}.

\bibitem{pan.potts}
D.~Panchenko.
\newblock Free energy in the {P}otts spin glass.
\newblock {\em Ann. Probab.}, 46(2):829--864, 2018.

\bibitem{pan.vec}
D.~Panchenko.
\newblock Free energy in the mixed {$p$}-spin models with vector spins.
\newblock {\em Ann. Probab.}, 46(2):865--896, \noop{2019}2018.

\end{thebibliography}

\end{document}